\documentclass[12pt,oneside]{article}
\usepackage[nottoc]{tocbibind}
\pdfoutput=1
\usepackage[latin1]{inputenc}
\usepackage{url}
\usepackage{amsmath}
\usepackage[colorlinks=true, urlcolor=blue,  linkcolor=blue, citecolor=blue]{hyperref}
\usepackage{amsfonts}
\usepackage{setspace}
\usepackage{amssymb}
\usepackage{graphicx}
\usepackage{amsthm}
\usepackage{makeidx}
\usepackage{enumitem}
\usepackage{subfigure}
\newtheorem{theorem}{Theorem}[section]
\newtheorem{lemma}[theorem]{Lemma}
\newtheorem{proposition}[theorem]{Proposition}
\newtheorem{corollary}[theorem]{Corollary}
\newtheorem{prop}[theorem]{Proposition}

  \renewenvironment{thebibliography}[1]{    \begin{oldthebibliography}{#1}      \setlength{\parskip}{-0.6ex}      \setlength{\itemsep}{0.2ex}  }  {    \end{oldthebibliography}}

\theoremstyle{definition}
\newtheorem{example}[theorem]{Example}
\theoremstyle{remark}
\newtheorem{remark}[theorem]{Remark}
\theoremstyle{plain}
\newenvironment{definition}[1][Definition.]{\begin{trivlist}
\item[\hskip \labelsep {\bfseries #1}]}{\end{trivlist}\par}

\newcommand{\C}{\mathbb C}
\newcommand{\R}{\mathbb R}
\newcommand{\RC}{\mathbb K}

\newcommand{\N}{\mathbb N}

\newcommand{\con}[1]{\overline{#1}}
\newcommand{\pd}[2]{\frac{\partial #1}{\partial #2}}
\newcommand{\clos}[1]{\overline{#1}}
\newcommand{\bd}{\partial}
\newcommand{\inter}[1]{\mathring{#1}}

\newcommand{\form}[1]{\Omega^{#1}}

\newcommand{\ball}{B}
\newcommand{\pdisc}{\Delta}
\newcommand{\pmetric}{\rho}
\newcommand{\emetric}{d}

\newcommand{\hol}{H}
\newcommand{\cts}{C}

\newcommand{\ccts}{C_{\mathbb C}}
\renewcommand{\sp}[2]{\left<#1,#2\right>}
\newcommand{\sprod}[2]{\sp{#1}{#2}}
\newcommand{\dprod}[2]{#1\cdot #2}
\newcommand{\grad}[1]{#1'}
\newcommand{\cgrad}[1]{\nabla #1}

\newcommand{\dil}[2]{#1_{(#2)}}
\makeatletter
\def\blfootnote{\xdef\@thefnmark{}\@footnotetext}
\makeatother

\makeindex
\title{\bf{The Levi Problem in $\C^n$: A Survey}
}
\author{\normalsize {Harry J. Slatyer}}
\date{}
\begin{document}
\maketitle
\begin{quotation} 
{\small \sl \noindent
We discuss domains of holomorphy and several notions of pseudoconvexity (drawing parallels with the corresponding notions from geometric convexity), and present a mostly self-contained solution to the Levi problem. We restrict
our attention to domains of $\C^n$.
}
\end{quotation}
\section{Introduction}
\label{chap:intro}
\subsection{Content}
The purpose of this paper is to provide a reasonably self-contained exposition of a solution to the so-called Levi problem, together with a comprehensive and detailed survey of some related classical concepts. We include
almost all of the interesting proofs in full generality, without any additional smoothness or boundedness assumptions. The only major exception is the proof of H\"ormander's theorem on solvability
of the inhomogeneous Cauchy-Riemann equations \cite{hormander}, which we omit in view of length restrictions and the fact that many self-contained proofs may be found elsewhere (see \cite[chapter 4]{krantz},
for example). 
We take care to motivate the definitions of holomorphic convexity and pseudoconvexity by
first discussing the analogous notions from geometric convexity, with the intention of providing the reader with some intuitive understanding of these properties. Thus, this paper is targeted towards
students specialising in complex analysis and geometry, and also professional mathematicians wishing to broaden their knowledge of these areas. 

Given a holomorphic function on some domain (a connected open set), it is natural to ask whether there exists a holomorphic function defined on a larger domain which agrees with the original
function on its domain -- that is, we seek a \emph{holomorphic extension} of the original function. In some cases the original domain can be such that \emph{any} holomorphic function necessarily admits
a holomorphic extension to a strictly larger domain.
For instance, in 1906 Hartogs showed that any function holomorphic on a domain of $\C^n$ (with $n\geq 2$) obtained by removing a compact set from another domain extends to a function
holomorphic on
the larger domain \cite{hartogs1906} -- note the contrast to the $1$-dimensional case, where there exist holomorphic functions with compactly contained singularities ($z\mapsto z^{-1}$, for example).
 This type
of result begs the following question: which domains have the property that any holomorphic function defined on these domains necessarily admits a holomorphic extension to a larger domain? Traditionally
we actually pose the inverse problem,
and ask on which domains there exists a holomorphic function that \emph{does not} extend holomorphically to points outside the domain. Such a domain is known as a \emph{domain of holomorphy},
because it is the most natural domain of existence of some holomorphic function. One may verify that every domain of the complex plane is a domain of holomorphy -- that is, given any domain
in $\C$ there exists a function holomorphic on that domain which does not extend to a holomorphic function on any larger domain. Upon passing to multiple complex variables however, this is no longer
the case (as shown by Hartogs' theorem on removal of compact singularities, for example).

Unfortunately it is typically difficult to verify directly from the definition whether a given domain is a domain of holomorphy, so it is desirable to obtain a more easily verified equivalent condition.
Such a condition is provided by the solution to the so-called Levi problem. Named for Levi's pioneering work in his 1911 paper \cite{levi11}, the Levi problem
is to show that the domains of holomorphy are precisely the \emph{pseudoconvex} domains. Pseudoconvexity, a local property of domains
which generalises the notion of convexity, is typically more easy to directly verify than whether a domain is a domain of holomorphy.
For domains with twice-differentiable boundaries there is an equivalent notion of pseudoconvexity, known as \emph{Levi pseudoconvexity}, which is particularly simple to verify for many
domains \cite{levi,levi11}.
It is relatively easy to show that domains of holomorphy are pseudoconvex, and that for domains with twice-differentiable boundaries Levi pseudoconvexity
is equivalent to pseudoconvexity. The problem of showing that the
(Levi) pseudoconvex domains are domains of holomorphy, which completes the solution to the Levi problem, is much more difficult -- the condition of Levi pseudoconvexity was introduced in 1910-1911 \cite{levi,levi11}, and 
was not proved
to be sufficient for a domain to be a domain of holomorphy until 1942 when Oka demonstrated the fact for $2$-dimensional space \cite{oka42}. For arbitrary dimensions 
the result was not obtained until 1953-1954, when Bremermann \cite{bremermann}, Norguet \cite{norguet} and Oka \cite{oka}  all presented independent proofs.

In this paper we discuss the above concepts in detail and present
a solution to the Levi problem.
We use definitions and methods from Boas \cite{boas}, H\"ormander \cite{hormanderbook}, Krantz \cite{krantz}, Range \cite{range}, Shabat \cite{itca} and Vladimirov \cite{vlad}
to show that domains of holomorphy are pseudoconvex, and that the various types of pseudoconvex domains are identical. To show that pseudoconvex domains are domains of holomorphy we follow the general method
of Oka \cite{oka}, which is to show first that strictly pseudoconvex domains are domains of holomorphy 
and then use the Behnke-Stein theorem \cite{behnkestein} together with the fact that pseudoconvex domains may be approximated
by an increasing sequence of strictly pseudoconvex domains \cite{lelong,oka}. 
While Oka shows the first step using his principle for joining two domains of holomorphy, we instead follow the argument of Boas \cite{boas} and apply
H\"ormander's theorem on solvability of the inhomogeneous Cauchy-Riemann equations \cite{hormander}. 

\subsection{Notation}
\label{blabs}
Before commencing in earnest our discussion of the Levi problem we list some notation and terminology that will occur frequently throughout the paper:
\begin{description}[font=\normalfont]
\item[]{$\R$, $\C$, $\N$, $\N_0$ -- the sets of real numbers, complex numbers, positive integers and non-negative integers}
\item[]{$\Re(z)$, $\Im(z)$ -- the real and imaginary parts of $z\in\C$}
\item[]{$|x|$, $\emetric(x,y)$, $x\cdot y$ -- the Euclidean norm, metric and dot product for $x,y\in\R^n$}
\item[]{$|z|$, $\|z\|$, $\emetric(z,w)$, $\pmetric(z,w)$, $\sprod{x}{y}$ -- the Euclidean norm, $L^\infty$ norm, Euclidean metric, $L^\infty$ metric and scalar product for $z,w\in\C^n$}
\item[]{$\ball(x,r)$, $\ball(z,r)$, $\pdisc(z,s)$ -- the ball of radius $r>0$ about $x\in\R^n$, about $z\in\C^n$ (with respect to the Euclidean metric), and the polydisc of (vector) radius $s$ about
$z$}
\item[]{$S^c, \clos{S}, \bd S, \inter{S}, S_{(r)}$ -- for a subset $S$ of some metric space, the complement, closure, boundary, interior and $r$-enlargement (where $r>0$) of $S$ (note that for subsets of $\C^n$, enlargements are always with respect to the $L^\infty$ metric)}
\item[]{$\emetric(x,S)$, $\emetric(T,S)$ -- for a point $x$ and subsets $S$ and $T$ of some space with metric $d(\cdot,\cdot)$, the distance from $x$ to $S$ ($\inf_{s\in S}\{d(x,s)\}$) and
the distance from $T$ to $S$ ($\inf_{t\in T}\{\emetric(t,S)\}$)}
\item[]{Neighbourhood of a point or set -- an open set containing that point or set}
\item[]{$\|f\|_B$ -- for a function $f\colon A\to\RC$ (where $\RC$ is a normed space) and subset $B\subset A$, the supremum of the norm of $f$ on $B$ (note that for functions mapping into $\R$ or $\C$, the Euclidean
norm is used)} 
\item[]{$\hol(C)$ -- the set of functions holomorphic on a set $C\subset\C^n$ (note that for a set which is not open, a function is holomorphic on that set if it is holomorphic on some neighbourhood of the set)}
\item[]{$\cts^k$, $\ccts^k$ -- the classes of $k$-times continuously differentiable real-valued functions and $k$-times continuously differentiable complex valued functions (when regarded as maps between real spaces)}
\item[]{$\grad{f}$, $\cgrad{g}$ -- the derivative (gradient) of $f\colon U\subset\R^n\to\R$ and $g\colon V\subset\C^n\to\C$}
\item[]{$\Delta_\delta f(z)$ -- the Levi form of $f\colon U\subset\C^n\to\R$ (see Proposition~\ref{leviformpsh})}
\end{description}

\section{Convexity}
\label{chap:convexity}
As described in Section~\ref{chap:intro}, the domains of holomorphy are precisely the domains
which satisfy any of several more general analogues of geometric convexity. In this preliminary section we present
several definitions of geometric convexity in the interests of motivating the generalised definitions which will appear later in the exposition. We omit the proofs in this section since they
are neither particularly difficult nor relevant.
\subsection{Convex sets}
We first give the usual definition:
\begin{definition}
A set $S\subset\R^n$ is \textbf{convex} if for every pair of points $x,y\in S$ the segment $\{x+t(y-x)\colon t\in [0,1]\}$ is a subset of $S$.\index{convex set}
\end{definition}
For example, open and closed balls in $\R^n$ are convex. It is immediate from the definition that the intersection of a collection of convex sets is convex, and that the union of an increasing sequence
of convex sets (that is, a sequence $\{S_j\}_{j\geq 1}$ such that $S_j\subset S_{j+1}$ for all $j\geq 1$) is convex. 
A less trivial property which will be useful later is the following:
\begin{proposition}
\label{suphyp}
Let $U\subset\R^n$ be open and convex. Then for all $a\in\bd U$ there exists a supporting hyperplane $P$ for $U$ containing $a$ (that is, $P$ is a hyperplane such that $a\in P$ and $P\cap U=\emptyset$).
\end{proposition}
Another intuitive fact is that convexity of a domain is determined locally at the boundary:
\begin{proposition}
\label{conlocal}
Let $U\subset\R^n$ be a domain. Then $U$ is convex if and only if for each $a\in\bd U$ there is a neighbourhood $V\subset\R^n$ of $a$ such that $U\cap V$ is convex.
\end{proposition}
\subsection{Convex hulls}
We will find that convexity of a domain may be determined by considering the convex hulls of its compact subsets, and this fact will allow us to generalise convexity in Section~\ref{chap:hol}. 
\begin{definition}
Let $C\subset\R^n$ be closed. The \textbf{convex hull} of $C$ is the intersection of all closed convex subsets of $\R^n$ containing $C$, and is denoted $\hat C_L$.\index{convex hull}
\end{definition}
\begin{proposition}
\label{conholcon}
Let $U\subset\R^n$ be open. Then $U$ is convex if and only if for all compact $K\subset U$ the convex hull $\hat K_L$ is a subset of $U$.
\end{proposition}
\begin{remark}
Basic topology, together with the observation that convex hulls of compact sets are compact and connected, implies that the condition $\hat K_L\subset U$ in the above proposition
is equivalent to requiring that $\hat K_L\cap U$ be compact.
\end{remark}
To use this fact to generalise convexity we will need a different description of the convex hull of a compact set. Notice that the level sets of affine functions $\R^n\to\R$ are the hyperplanes, and the pullbacks of intervals of the form $(-\infty,t]$ and $[t,\infty)$ (where $t\in\R$) are the closed half-spaces.
Along with the observation that the convex hull of a closed set is the intersection of all closed half-spaces containing the set, this yields:
\begin{proposition}
Let $K\subset\R^n$ be compact. Then the convex hull $\hat K_L$ is given by
\begin{equation*}\label{convexhulleqdef} \hat K_L=\{x\in\R^n\colon |A(x)|\leq \|A\|_K \mbox{ for all affine $A\colon\R^n\to\R$}\}. \end{equation*}
\end{proposition}
With the remark following Proposition~\ref{conholcon} in mind this implies the following, which will motivate the definition of holomorphic convexity in Section~\ref{chap:hol}:
\begin{corollary}
\label{congen}
Let $U\subset\R^n$ be open. Then $U$ is convex if and only if for all compact $K\subset U$,
$$ \{x\in U\colon |A(x)|\leq \|A\|_K\mbox{ for all affine $A\colon\R^n\to\R$}\} =\hat K_L\cap U $$
is a compact set.
\end{corollary}
\subsection{The continuity principle}
In this subsection we introduce the ``continuity principle'' which, roughly speaking, 
states that when a sequence of line segments converges to some limit set, and the limits of the sequences of endpoints are in the
domain, then the entire limit set is also in the domain. 
To state the continuity principle precisely we require some notation and a definition.
If $L\subset\R^n$ is a closed line segment we let $\bd L$ denote the set consisting of its two endpoints (rather than the topological boundary). 
\begin{definition}
Let $\{S_j\}_{j\geq 1}$ be a sequence of subsets of a metric space $\RC$ and let $S\subset \RC$. Then $\{S_j\}_{j\geq 1}$ \textbf{converges to $S$} if for all $\epsilon>0$ there exists
$J\geq 1$ such that whenever $j\geq J$ we have $S_j\subset \dil{S}{\epsilon}$ and $S\subset \dil{(S_j)}{\epsilon}$ (where the subscript $\epsilon$ is for the $\epsilon$-dilation). In
this case we write $S_j\to S$ or $\lim_{j\to\infty} S_j=S$.
\end{definition}
Now we can define the continuity principle formally:
\begin{definition}
Let $U\subset\R^n$ be a domain. Then $U$ \textbf{satisfies the continuity principle} if, for every sequence of closed line segments $\{L_j\}_{j\geq 1}$ satisfying $L_j\subset U$ for all $j\geq 1$,
$\bd L_j\to B$ and $L_j\to L$ where $B\subset U$ and $L\subset\R^n$ are compact, we have $L\subset U$.\index{continuity principle}
\end{definition}
\begin{proposition}
A domain $U\subset\R^n$ is convex if and only if it satisfies the continuity principle.
\end{proposition}
\subsection{Convex exhaustion functions}
In this subsection we observe that the convex domains are those on which there is a convex function (to be defined shortly) whose sublevel sets are all compact.
\begin{definition}
Let $S\subset\RC$ (where $\RC$ is a normed space) be a set. A function $f\colon S\to \R$ is an \textbf{exhaustion function} for $S$ if for all $r\in\R$, $f^{-1}((-\infty,r])$ is compact.\index{exhaustion function}
\end{definition}
It is obvious from the definition that continuous exhaustion functions for domains are those which tend to $\infty$ at each boundary point:
\begin{prop}
\label{exasdf}
Let $U\subset\RC$ (where $\RC$ is a normed space) be a domain and $f\colon U\to\R$ a continuous function. Then $f$ is an exhaustion function for $U$ if and only if whenever $\{x_j\}_{j\geq 1}\subset U$
is a sequence with $x_j\to a\in\bd U$ or $x_j\to\infty$ (that is, $|x_j|\to +\infty$) we have $f(x_j)\to+\infty$.
\end{prop}
\begin{example}
\label{exex}
Let $U\subset\C^n$ be a domain. Let $f\colon U\to\R$ be given by $f(z):=|z|^2$ if $\bd U=\emptyset$ and $f(z):=|z|^2-\ln \emetric(z,\bd U)$ otherwise (where $\emetric(\cdot,\cdot)$ denotes the Euclidean metric).
We claim that $f$ is an exhaustion function for $U$. If $\bd U=\emptyset$ this is trivial, so consider the case when $\bd U\neq\emptyset$.
If $z_j\to a\in\bd U$ then $\{|z_j|^2\}_{j\geq 1}$ is bounded and $\ln\emetric(z_j,\bd U)\to-\infty$, so $f(z_j)\to+\infty$.
If $z_j\to\infty$ then for sufficiently large $j$ we have $\emetric(z_j,\bd U)\leq 2|z_j|$ and thus $f(z_j)\geq |z_j|^2-\ln |z_j|-\ln 2$, which implies $f(z_j)\to+\infty$. 
\end{example}
Thus every domain admits an exhaustion function. To characterise convexity in terms of the existence of such functions we introduce a particular class of functions:
\begin{definition}
A function $f\colon U\subset\R\to\R$ (where $U$ is an open interval) is \textbf{convex} if for all $a,b\in U$ with $a<b$ and $x\in [a,b]$ we have $f(x)\leq f(a)+(f(b)-f(a))\frac{x-a}{b-a}$. A function
$g\colon V\subset\R^n\to\R$ (where $V$ is a domain) is \textbf{convex} if for all $a\in V$ and $\delta\in\R^n$ with $|\delta|=1$ the function $x\mapsto g(a+\delta x)$ is convex on each component
of $\{x\in\R\colon a+\delta x\in V\}$.\index{convex function} 
\end{definition}

The convex functions of one variable are those which satisfy the property that for any two points on the graph of the function, the graph lies below
the straight line between those points.
The convex functions of several variables are those whose restriction to any line segment is convex. 
It may be verified that convex functions are continuous (see \cite[page 85]{vlad}). 

For twice-differentiable functions we have another condition for convexity:
\begin{prop}
\label{convexdiff}
Let $U\subset\R^n$ be a domain and $f\colon U\to\R$ a $\cts^2$ function. Then $f$ is convex if and only if for all $\delta=(\delta_1,\dots,\delta_n)\in\R^n$ and $x\in U$ we have
$$ \Delta_\delta f(x):= \sum_{j,k=1}^n \frac{\partial^2 f}{\partial x_j\partial x_k}\bigg|_x\delta_j\delta_k\geq 0.$$
\end{prop}
This result implies the following, which
motivates the definition of pseudoconvexity in Section~\ref{chap:pseudo}:
\begin{proposition}
\label{convexex}
A domain $U\subset\R^n$ is convex if and only if there exists a convex exhaustion function for $U$.
\end{proposition}
\subsection{Convexity of domains with twice-differentiable boundaries}
Until now we have considered arbitrary domains of $\R^n$, but for those domains with twice-differentiable boundaries we have an additional local characterisation of convexity.
\begin{definition}
Let $U\subset\R^n$ be a domain and $a\in\bd U$. If $V\subset\R^n$ is a neighbourhood of $a$ and $f\colon V\to\R$ is such that $U\cap V=f^{-1}((-\infty,0))$ then
$f$ is a \textbf{local defining function} for $U$ at $a$.\index{local defining function}

Let $k\geq 1$ (we allow $k=\infty$). The domain $U$ is said to have \textbf{$k$-times differentiable boundary} or \textbf{$\cts^k$ boundary} if at each boundary point there is a $\cts^k$ local defining function whose gradient is non-zero 
at the point.
If $a\in\bd U$ and $f\colon V\to\R$ is a $\cts^k$ local defining function for $U$ at $a$ then the \textbf{tangent space} to $\bd U$ at $a$ with respect to $f$ is
$ T_a(f):=\{\delta\in\R^n\colon \dprod{\delta}{\grad{f}(a)}=0\}.$\index{tangent space}\index{differentiable boundary}
\end{definition}
\begin{remark}
\label{defrem}
By the implicit function theorem a $k$-times differentiable boundary is locally given by the graph of a real-valued $\cts^k$ function.
More precisely, consider $a\in\R^n$ and a $\cts^k$ function $f\colon V\to\R$ (where $V$ is a neighbourhood of $a$) which is zero at $a$ and non-degenerate,
and assume (without loss of generality) that $\pd{f}{x_n}\big|_a>0$. By the implicit function theorem there are connected open sets $W\subset\R^{n-1}$
and $I\subset\R$ such that $a\in W\times I\subset V$ and a $\cts^k$ function $g\colon W\to I$ such that $f(x)=0$ if and only if $x_n=g(x_1,\dots,x_{n-1})$
for $x\in W\times I$. Moreover, using the fact that $f$ is differentiable, if $x_n>g(x_1,\dots,x_{n-1})$ then $f(x)>0$ and if $x_n<g(x_1,\dots,x_{n-1})$ then $f(x)<0$.
In particular, if $f$ is a defining function for a domain $U$ at $a\in\bd U$ then after replacing $V$
with the subset $W\times I$ we have $U\cap V=f^{-1}((-\infty,0))=\{x\in V\colon x_n<g(x_1,\dots,x_{n-1})\}$, which implies
\begin{equation} \bd U\cap V=\{x\in V\colon x_n=g(x_1,\dots,x_{n-1})\}=f^{-1}(\{0\}).\label{defnais}\end{equation}
It is also clear that $x\mapsto x_n-g(x_1,\dots,x_{n-1})$ is a $\cts^k$ local defining function for $U$ at $a$.
We will use these observations in Section~\ref{chap:pseudo}.
\end{remark}
One may verify that convexity of a domain with twice-differentiable boundary is determined by the curvatures of its local defining functions:
\begin{proposition}
\label{blakeyy}
Let $U\subset\R^n$ be a domain with $\cts^2$ boundary. Then $U$ is convex if and only if for all $a\in\bd U$ there is a $\cts^2$ local defining function $f\colon V\to\R$ (where $V$ is a neighbourhood of $a$)
such that $f'(a)\neq 0$ and $\Delta_\delta f(a)\geq 0$ for all $\delta\in T_a(f)$. 
\end{proposition}

\section{Domains of holomorphy}
\label{chap:hol}
In this section we define domains of holomorphy, discuss some examples and conditions, and then introduce the notion of
holomorphic convexity and show that the domains of holomorphy are precisely the holomorphically convex domains. Using this we will prove several properties of domains of holomorphy. The material
in this section is drawn from
\cite{krantz}, \cite{range} and \cite{itca}.

\subsection{Domains of holomorphy}
We must first make precise the notion of holomorphic extension:
\begin{definition}
Let $U\subset\C^n$ be open, $f\in \hol(U)$ and $a\in U^c$ (where $\hol(U)$ denotes the set of holomorphic functions on $U$). Then $f$ \textbf{extends holomorphically} to $a$ if there is a connected open neighbourhood $V$ of $a$ and a holomorphic function $g\in H(V)$
such that $f\equiv g$ on a non-empty open subset of $U\cap V$.\index{holomorphic extension}
\end{definition}
Note that we do \emph{not} require that the functions $f$ and $g$ agree on all of $U\cap V$, so it may not be the case that $f$ extends to a function holomorphic on $U\cup V$.
That is, if $f\in \hol(U)$ extends holomorphically to $a\in U^c$ then there is not necessarily a domain $W\subset\C^n$ and function $g\in H(W)$ such that $U\subset W$, $a\in W$ and $f\equiv g$ on $U$.
We use this definition of holomorphic extension to avoid the need for multi-valued extensions. For example, the function $z\mapsto \sqrt{z}$
defined on the domain $U:=\C\setminus\{x\colon x\in\R,x\geq 0\}$ extends holomorphically to the point $z=1$, but there is no way to define this function holomorphically on a domain of $\C$ which
contains both $U$ and the point $z=1$ unless we allow the function to take multiple values. 
\begin{definition}
Let $U\subset\C^n$ be a domain. Then $U$ is a \textbf{domain of holomorphy} if there exists a function $f\in \hol(U)$ which does not extend holomorphically to any point of $U^c$.\index{domain of holomorphy}
An open set $V\subset\C^n$ is an \textbf{open set of holomorphy} if each component of $V$ is a domain of holomorphy.\index{open set of holomorphy}
\end{definition}
We introduce a particular class of functions:
\begin{definition}
Let $U\subset\C^n$ be a domain and $f\in\hol(U)$ a function.
If for every domain $V$ intersecting $\bd U$ and every component $W$ of $U\cap V$
there is a sequence $\{z_j\}_{j\geq 1}\subset W$ with $f(z_j)\to\infty$ then we say that $f$ is \textbf{essentially unbounded} on $\bd U$.\index{essentially unbounded function}
\end{definition}
The following demonstrates the importance of these functions:
\begin{prop}
\label{unbdedhol}
Let $U\subset\C^n$ be a domain and $f\in \hol(U)$ a function essentially unbounded on $\bd U$. Then $U$ is a domain of holomorphy.
\end{prop}
\begin{proof}
Suppose, for a contradiction, that $f$ extends holomorphically to $b\in U^c$, so there is a domain $V\subset\C^n$ containing $b$ and a function $g\in H(V)$ such that
$f\equiv g$ on a non-empty open set $X\subset U\cap V$. Let $W'$ be a component of $U\cap V$ containing a non-empty open subset of $X$, and let $a\in \bd W'\cap V$.
Note that $a\in\bd U$ because $\bd W'\subset \bd U\cup\bd V$ and obviously $a\not\in\bd V$.
Now let $r>0$ such that $\clos{\ball(a,r)}\subset V$ (where $\ball(a,r)$ denotes the open ball of radius $r$ about $a$) and let $W$ be a component of $U\cap\ball(a,r)$ contained in $W'$. Since
$f$ is essentially unbounded on $\bd U$ there is a sequence $\{z_j\}_{j\geq 1}\subset W\subset W'$ with $f(z_j)\to\infty$. By the uniqueness theorem it follows that $f(z_j)=g(z_j)$ for each $j\geq 1$ and thus
$g(z_j)\to\infty$. But $\clos{W}$ is a compact subset of $V$,
so $g$ attains a maximum modulus on $\clos{W}$, which is a contradiction. Therefore $f$ does not extend holomorphically to any point of $U^c$, so $U$ is a domain of holomorphy.
\end{proof}
In the next subsection we will prove the following, which strengthens Proposition~\ref{unbdedhol} and is extremely useful in practice:
\begin{theorem}
\label{usefuldh}
Let $U\subset\C^n$ be a domain and suppose for each $a\in\bd U$ there is a function $f_a\in\hol(U)$ tending to $\infty$ at $a$ (that is,
for every sequence $\{z_j\}_{j\geq 1}\subset U$ with $z_j\to a$ we have $f_a(z_j)\to\infty$). Then $U$ is a domain of holomorphy.
\end{theorem}

We consider some consequences of this theorem.
\begin{corollary}
\label{mira}
Let $U\subset\C^n$ be a domain and suppose for each $a\in \bd U$ there is a non-vanishing function $g_a\in\hol(U)$ which tends to zero at $a$.
Then $U$ is a domain of holomorphy.
\end{corollary}
\begin{proof} Apply Theorem~\ref{usefuldh} to the reciprocals of the functions $g_a$.
\end{proof}
\begin{example}
Let $U\subset\C$ be a domain. If for each $a\in\bd U$ we consider the function $z\mapsto z-a$ then the hypothesis of Corollary~\ref{mira} is satisfied, so $U$ is a domain of holomorphy.
Thus all domains of $\C$ are domains of holomorphy.
\end{example}
\begin{example}
Let $U\subset\C^n$ be a convex domain (that is, $U$ is convex when regarded as a subset of $\R^{2n}$).
Let $a\in\bd U$, so 
there exists a supporting hyperplane for $U$ at $a$. That is, there exists $\delta\in\C^n$ such that whenever $\Re\sp{z-a}{\delta}=0$ we have $z\not\in U$ (where we have recalled that the
real part of the inner product $\sp{\cdot}{\cdot}$ on $\C^n$ gives the dot product when the vectors are regarded as elements of $\R^{2n}$). Thus the function
$z\mapsto \sp{z-a}{\delta}$ is holomorphic and non-vanishing on $U$ and tends to zero at $a$. By Corollary~\ref{mira} it follows that $U$ is a domain of holomorphy.
Therefore the convex domains (in particular open balls) are domains of holomorphy.
\end{example}

Next we construct a domain which is not a domain of holomorphy.
\begin{prop}
\label{reinloghol}
Let $U\subset\C^n$ be a complete Reinhardt domain with center $0$ which is not logarithmically convex. Then $U$ is a not domain of holomorphy.
\end{prop}
We recall that a complete Reinhardt domain with center $0$ is a domain which can be written as a union of polydiscs
centered at $0$, and a domain $V\subset\C^n$ is logarithmically convex if its logarithmic image $\{(\ln|z_1|,\dots,\ln|z_n|)\colon z\in V,z_1\dots z_n\neq 0\}$ is convex.
One may verify that a function holomorphic in a complete Reinhardt domain with center $0$ has a power series representation about $0$ in the entire domain, and that the domain of convergence of any power series
about $0$ is logarithmically convex (see \cite[subsection 7]{itca}).
\begin{proof}[Proof of Proposition~\ref{reinloghol}]
Let $f\in\hol(U)$. We know that there is a power series representation for $f$ about $0$ in $U$, and that the domain of convergence $V$ of this series is a logarithmically convex
domain containing $U$. Since $U$ is not logarithmically convex we have that $U\neq V$, and $f$ extends holomorphically to each point of $V\setminus U$.
\end{proof}
\begin{example}
With $e$ denoting Euler's number, 
let $U:=\pdisc(0,(e,e^2))\cup \pdisc(0,(e^2,e))\subset\C^2$ (where $\pdisc(z,r)$ denotes the polydisc of (vector) radius $r$ about $z\in\C^n$), so clearly $U$ is a complete Reinhardt domain with center $0$. 
The logarithmic image of $U$ is
$$
 \{(\ln|z_1|,\ln|z_2|)\colon (z_1,z_2)\in U,\,z_1z_2\neq 0\}= 
	[(-\infty,1)\times (-\infty,2)]\cup [(-\infty,2)\times (-\infty,1)],$$
which is not convex. Thus $U$ is not logarithmically convex. It follows from Proposition~\ref{reinloghol} that $U$ is not a domain of holomorphy.
\end{example}
We have defined domains of holomorphy and given some examples and a non-example, but currently we have no way to describe domains of holomorphy in simple geometric terms, and we have
yet to prove the important Theorem~\ref{usefuldh}.
\subsection{Holomorphic convexity}
We showed above that convex domains of $\C^n$ are domains of holomorphy, but clearly the converse is not true (take any non-convex domain of $\C$, for instance). Thus if we wish 
to describe domains of holomorphy with some notion of ``convexity'' we must use a more general definition. 
With Corollary~\ref{congen} in mind we introduce the following definition, where we simply replace the affine functions with holomorphic functions:
\begin{definition}
Let $U\subset\C^n$ be a domain and $K\subset U$ compact. The \textbf{holomorphically convex hull} of $K$, denoted $\hat K$, is
$$ \hat K:=\{z\in U\colon |f(z)|\leq \|f\|_K \mbox{ for all } f\in H(U)\}=\bigcap_{f\in H(U)} |f|^{-1}([0,\|f\|_K]). $$
The domain $U$ is said to be \textbf{holomorphically convex} if for all compact $K\subset U$ the holomorphically convex hull $\hat K$ is compact. 
\index{holomorphically convex hull}
\index{holomorphically convex domain}
\end{definition}
Note that the holomorphically convex hull of a particular compact set depends on the domain of which the compact set is a subset, so if this is not clear from the context
we may write $\hat K_{\hol(U)}$ to explicitly denote the holomorphically convex hull of $K$ as a subset of $U$. 
We now briefly note some properties of holomorphically convex hulls. Obviously the holomorphically convex hull of a compact set always contains the original compact set.
Notice also that since the identity map $f(z)=z$ is holomorphic, if $K$ is some compact subset of a domain and $\hat K$ is the holomorphically convex
hull then for any $z\in\hat K$ we have $|z|=|f(z)|\leq \|f\|_K$, so holomorphically convex hulls are always bounded. Furthermore, since holomorphically convex
hulls are closed in the subspace topology (this is obvious from the definition) we see that 
a holomorphically convex hull is compact if and only if it is a positive distance from the boundary of the domain, and this is true if and only if it is contained inside a compact subset of the domain. 
It is also evident from the definition 
that for any function $f\in\hol(U)$ and compact $K\subset U$
we have $\|f\|_K=\|f\|_{\hat K}$.
This immediately implies that for a compact $K\subset U$ we have $\hat{\hat K}=\hat K$ (provided $\hat K$ is compact).

We have the following relationship between domains of holomorphy and holomorphically convex domains:
\begin{theorem}
\label{hcdh}
If $U\subset\C^n$ is holomorphically convex then $U$ is a domain of holomorphy.
\end{theorem}
To prove this we will require an intermediate result:
\begin{lemma}
Let $U\subset\C^n$ be holomorphically convex. Then there is a sequence $\{K_j\}_{j\geq 1}$ of compact subsets of $U$ such that $K_j\subset\inter K_{j+1}$ and
$K_j=\hat K_j$ for all $j\geq 1$, and $U=\bigcup_{j\geq 1} \inter K_j$.\end{lemma}
\begin{proof}
For $j\geq 1$ define the compact set $L_j:=\{z\in U\colon \emetric(z,\bd U)\geq 1/j \text{ and } |z|\leq j\}$, so certainly $L_j\subset\inter L_{j+1}$ for all $j\geq 1$ and $U=\bigcup_{j\geq 1}\inter L_j$.
Let $K_1:=\hat L_1$, and observe that $K_1$ is a compact subset of $U$ with $K_1=\hat K_1$. Let $j_2>1$ be sufficiently large that $K_1\subset\inter L_{j_2}$ (this is possible
because $U=\bigcup_{j\geq 1}\inter L_j$), and let $K_2:=\hat L_{j_2}$, so $K_2$ is a compact subset of $U$ with $K_2=\hat K_2$, and $K_1\subset\inter L_{j_2}\subset \inter K_2$. Since
$j_2\geq 2$ we also have $\inter L_2\subset \inter K_2$.
Repeating this argument we obtain a sequence $\{K_j\}_{j\geq 1}$ of compact subsets with $K_j\subset\inter K_{j+1}$, $K_j=\hat K_j$ and $\inter L_j\subset \inter K_j$ for each $j\geq 1$. 
From the last property and the fact that $U=\bigcup_{j\geq 1}\inter L_j$ it follows that $U=\bigcup_{j\geq 1}\inter K_j$, so $\{K_j\}_{j\geq 1}$ is the required sequence.
\end{proof}
\begin{proof}[Proof of Theorem~\ref{hcdh}]
Note that if $U=\emptyset$ or $U=\C^n$ the assertion is trivial, so assume this is not the case.

In view of Proposition~\ref{unbdedhol}, to prove $U$ is a domain of holomorphy it is enough to find a function $f\in\hol(U)$ which is essentially unbounded on $\bd U$.
Let $\mathcal A:=\{a_k\}_{k\geq 1}\subset U$ be the countable set consisting of all points in $U$ with rational coordinates (that is, the real and imaginary parts of every component of every $a_k$ are rational),
and for each $k\geq 1$ let $B_k:=\ball(a_k,\emetric(a_k,\bd U))$ (note that $B_k\subset U$).
Now let $\{Q_j\}_{j\geq 1}$ be a sequence of elements of $\{B_k\}_{k\geq 1}$ such that every $B_k$ is given by $Q_j$ for
infinitely many indices $j$ (for example, let $Q_1:=B_1$, $Q_2:=B_1$, $Q_3:=B_2$, $Q_4:=B_1$, $Q_5:=B_2$, $Q_6:=B_3$, and so on). We will find a function $f\in\hol(U)$ and a sequence
$\{z_j\}_{j\geq 1}$ with $z_j\in Q_j$ for each $j\geq 1$ such that $f(z_j)\to\infty$ as $j\to\infty$. Suppose, for a moment, that we have found such a function and sequence.
Let $V$ be a domain intersecting $\bd U$ and suppose $W$ is a component of $U\cap V$. Let $a\in\bd W\cap V$ and note that we also have $a\in \bd U$. Let $r:=\emetric(a,\bd V)/2$, so
because
$\mathcal A$ is dense in $U$ we have $a_k\in W\cap\ball(a,r)$ for some $k\geq 1$. Clearly $\emetric(a_k,\bd U)<r<\emetric(a_k,\bd V)$, meaning
$B_k\subset U\cap V$ and thus $B_k\subset W$ (because $B_k$ is connected). We have $B_k=Q_{j_l}$ for infinitely many indices $j_1<j_2<\dots$, so $\{z_{j_l}\}_{l\geq 1}\subset B_k\subset W$
and $f(z_{j_l})\to\infty$. This argument applies for each component $W$ of each domain $V$ intersecting $\bd U$, so $f$ is essentially unbounded on $\bd U$.

It remains to find the function $f\in\hol(U)$ and sequence $\{z_j\}_{j\geq 1}$. Let $\{K_j\}_{j\geq 1}$ be the sequence of compact subsets of $U$ whose existence is asserted by the lemma. Passing
to a subsequence of $\{K_j\}_{j\geq 1}$ if necessary we may assume $Q_j\cap (K_{j+1}\setminus K_j)\neq\emptyset$ for all $j\geq 1$.
 Thus for all $j\geq 1$ there exists $z_j\in Q_j\cap (K_{j+1}\setminus K_j)$, and since
$z_j\not\in K_j=\hat K_j$ there exists $f_j\in\hol(U)$ such that $|f_j(z_j)|>\|f_j\|_{K_j}$, and scaling $f_j$ if necessary we may assume $|f_j(z_j)|>1\geq \|f_j\|_{K_j}$.

Let $p_1:=1$, and inductively choose $p_j\in\N$ sufficiently large that for all $j\geq 1$ we have
\begin{equation} \frac{1}{j^2}|f_j(z_j)|^{p_j}- \sum_{k=1}^{j-1} \frac{1}{k^2}|f_k(z_j)|^{p_k}\geq j \label{holmiraest}\end{equation}
(this is possible because $|f_j(z_j)|>1$). Now set, for all $z\in U$, 
$$ f(z):=\sum_{k=1}^\infty \frac{1}{k^2}f_k(z)^{p_k}. $$
Let $j\geq 1$, so if $z\in K_j$ then $z\in K_k$ for all $k\geq j$ and in particular $|f_k(z)|\leq 1$ for all $k\geq j$. By the Weierstrass $M$-test the series for $f$ 
converges uniformly on $K_j$,
so $f$ is holomorphic on $\inter K_j$. But $U=\bigcup_{j\geq 1}\inter K_j$, so $f$ is holomorphic on $U$.
For any $j\geq 1$ we have
$$ |f(z_j)| \geq \frac{1}{j^2}|f_j(z_j)|^{p_j}-\sum_{k=1}^{j-1} \frac{1}{k^2}|f_k(z_j)|^{p_k} -\sum_{k=j+1}^\infty \frac{1}{k^2}|f_k(z_j)|^{p_k}\geq j-\sum_{k=j+1}^\infty \frac{1}{k^2}\geq j-\frac{\pi^2}{6},$$
where for the second inequality we have used \eqref{holmiraest} and the fact that when $k>j$ we have $z_j\in K_k$ and thus $|f_k(z_j)|\leq 1$. Therefore $f(z_j)\to\infty$ as $j\to\infty$. From the earlier argument it follows that $f$ is essentially unbounded on $\bd U$ and thus $U$ is a domain of holomorphy.
\end{proof}
As a consequence of this we may prove Theorem~\ref{usefuldh}:
\begin{proof}[Proof of Theorem~\ref{usefuldh}]
To show $U$ is a domain of holomorphy it suffices, by Theorem~\ref{hcdh}, to show it is holomorphically convex. Let $K\subset U$ be compact, so we must show $\hat K$ is compact, and it suffices to show $\emetric(\hat K,\bd U)>0$.
Suppose this is not the case, so there is a sequence $\{z_j\}_{j\geq 1}\subset \hat K$ with $\emetric(z_j,\bd U)\to 0$. Since $\hat K$ is bounded its closure is compact,
so passing to a subsequence if necessary we may assume the sequence converges to a point $a$ in this closure, and because $\emetric(z_j,\bd U)\to 0$ we have $a\in\bd U$. By hypothesis
there is a function $f_a\in\hol(U)$ such that $f_a(z_j)\to\infty$. But $\|f_a\|_{\hat K}=\|f_a\|_K<\infty$ and $\{z_j\}_{j\geq 1}\subset \hat K$, yielding a contradiction.
Therefore $\emetric(\hat K,\bd U)>0$, so $\hat K$ is compact, meaning $U$ is holomorphically convex and is hence a domain of holomorphy.
\end{proof}

We will find that the converse to Theorem~\ref{hcdh} holds. We will need the following result on holomorphic extension to neighbourhoods
of holomorphically convex hulls:
\begin{prop}
\label{simultext}
Let $U\subset\C^n$ be a domain, let $K\subset U$ be a compact subset, let $r:=\pmetric(K,\bd U)$ (where $\pmetric(\cdot,\cdot)$ is the $L^\infty$ metric) and let $\hat K$ be the holomorphically convex hull of $K$.
Then every $f\in\hol(U)$ extends holomorphically to each point of the $r$-dilation $\dil{\hat K}{r}=\bigcup_{z\in\hat K}\pdisc(z,r)$.
\end{prop}
\begin{proof}
Let $f\in\hol(U)$ and $a\in\dil{\hat K}{r}$, so $a\in\pdisc(b,r)$ for some $b\in\hat K\subset U$. In a neighbourhood $W$ of $b$ we have the power series
$$ f(z) = \sum_{|J|=0}^\infty c_J(z-b)^J,\qquad c_J:=\frac{1}{J!}\frac{\partial^{|J|}f}{\partial z^J}\bigg|_b. $$
All partial derivatives of $f$ are holomorphic on $U$, and $b\in\hat K$, so for $J\in\N_0^n$ (where $\N_0$ is the set of non-negative integers) we have
$$ |c_J|=\frac{1}{J!}\left|\frac{\partial^{|J|}f}{\partial z^J}\bigg|_{\phantom{p}\!\!\!b}\right|\leq \frac{1}{J!}\left\|\frac{\partial^{|J|}f}{\partial z^J}\right\|_K. $$
Let $r'<r$, so $S:=\clos{\dil{K}{r'}}$ is a compact subset of $U$. For any  $p\in K$ we have $\clos{\pdisc(p,r')}\subset S\subset U$ and thus $f\in\hol(\clos{\pdisc(p,r')})$, and 
if we set $S_{p}:=\{z\in U\colon |z_k-p_k|=r',1\leq k\leq n\}$ then, since $S_p\subset S$, we have $\|f\|_{S_p}\leq \|f\|_S$. Expanding $f$
in a power series about any $p\in K$ and applying the Cauchy estimate:
$$ \frac{1}{J!}\left|\frac{\partial^{|J|}f}{\partial z^J}\bigg|_p\right|\leq \frac{\|f\|_{S_p}}{r'^{|J|}}\leq \frac{\|f\|_S}{r'^{|J|}} \implies \frac{1}{J!}\left\|\frac{\partial^{|J|}f}{\partial z^J}\right\|_K\leq \frac{\|f\|_S}{r'^{|J|}}. $$
Therefore, for all $J\in\N_0^n$,
$$ |c_J|\leq \frac{\|f\|_S}{r'^{|J|}} \implies |c_Jr'^{|J|}|\leq \|f\|_S,$$
so the power series for $f$ about $b$ converges in $\pdisc(b,r')$ (the terms of the series are bounded at the point $b+(r',\dots,r')$). 
This is true for all $r'<r$, so the series converges in $\pdisc(b,r)$. Since
a convergent power series is a holomorphic function it follows that this series yields a function holomorphic on $\pdisc(b,r)\ni a$ which agrees with $f$ on $W$. That is,
$f$ extends holomorphically to $a$.
\end{proof}
This immediately implies the following:
\begin{corollary}
If $U\subset \C^n$ is a domain of holomorphy then $U$ is holomorphically convex.
\end{corollary}

In particular, the domains of holomorphy and the holomorphically convex domains are identical -- this result is attributed to Cartan and Thullen \cite{cartanthullen}, and is usually known as the Cartan-Thullen
theorem. With this fact in mind we might hope that the domains of holomorphy obey some of the closure properties
of convex domains.
\begin{proposition}
\label{intscthol}
Let $\{U_\lambda\}_{\lambda\in\Lambda}$, where $\Lambda$ is some indexing set, be a set of domains of holomorphy in $\C^n$, and let $U$ be a connected component of the interior of 
$\bigcap_{\lambda\in\Lambda} U_\lambda$. Then $U$ is a domain of holomorphy.
\end{proposition}
To prove this we require a lemma:
\begin{lemma}
\label{hulldist}
A domain $U\subset\C^n$ is holomorphically convex if and only if for all compact $K\subset U$ the hull $\hat K$ satisfies $\pmetric(\hat K,\bd U)=\pmetric(K,\bd U)$.
\end{lemma}
\begin{proof}
If for every compact $K\subset U$ the hull $\hat K$ satisfies $\pmetric(\hat K,\bd U)=\pmetric(K,\bd U)$, then $\pmetric(\hat K,\bd U)>0$, so $\hat K$ is compact
and thus $U$ is holomorphically convex.

Conversely, suppose $U$ is holomorphically convex and let $K\subset U$ be compact, so $\hat K$ is compact. Clearly $K\subset\hat K$ so $\pmetric(\hat K,\bd U)\leq \pmetric(K,\bd U)$. 
Let $r:=\pmetric(K,\bd U)$. By Proposition~\ref{simultext} every function holomorphic on $U$ extends holomorphically
to each point of the $r$-dilation $\dil{\hat K}{r}$, and since $U$ is a domain of holomorphy it follows that $\dil{\hat K}{r}\subset U$. Thus $\pmetric(\hat K,\bd U)\geq r=\pmetric(K,\bd U)$,
so $\pmetric(\hat K,\bd U)=\pmetric(K,\bd U)$ as required.
\end{proof}
\begin{proof}[Proof of Proposition~\ref{intscthol}]
Let $K\subset U$ be compact. It suffices to show the holomorphically convex hull $\hat K$ is compact. Let $z\in\hat K$ and $\lambda\in\Lambda$, so for any $g\in\hol(U_\lambda)$ we have $|g(z)|\leq \|g\|_K$ (since $g\big|_U\in\hol(U)$), which implies $z\in\hat K_\lambda:=\hat K_{\hol(U_\lambda)}$.
Therefore $\hat K\subset\hat K_\lambda$ for all $\lambda\in\Lambda$, so
$\pmetric(\hat K,\bd U_\lambda)\geq \pmetric(\hat K_\lambda,\bd U_\lambda)=\pmetric(K,\bd U_\lambda)\geq \pmetric(K,\bd U)$ (using the lemma for the equality). Therefore
with $r:=\pmetric(K,\bd U)>0$ we have $\dil{\hat K}{r}\subset U_\lambda$ for all $\lambda\in\Lambda$, so $\dil{\hat K}{r}\subset \bigcap_{\lambda\in\Lambda} U_\lambda$.
Since $\dil{\hat K}{r}$ is open it is actually contained in the interior of this intersection, and since $\hat K$ is in the single component $U$ of this interior
we must have $\dil{\hat K}{r}\subset U$. That is, $\pmetric(\hat K,\bd U)\geq r>0$, so $\hat K$ is compact and thus $U$ is holomorphically convex.
\end{proof}
\begin{proposition}
\label{prodhol}
Let $U_1\subset\C^n$ and $U_2\subset\C^m$ be domains of holomorphy. Then the Cartesian product $U:=U_1\times U_2\subset\C^{n+m}$ is a domain of holomorphy.
\end{proposition}
\begin{proof}
Let $K\subset U_1\times U_2$ be compact, so it suffices to show the holomorphically convex hull $\hat K$ is compact. Let $\pi_1\colon\C^{n+m}\to\C^n$ be defined
by $\pi_1(z_1,z_2)=z_1$ (where $z_1\in\C^n$, $z_2\in\C^m$) and set $K_1:=\pi_1(K)$, and define $\pi_2\colon\C^{n+m}\to\C^m$ and $K_2:=\pi_2(K)$ analogously. Then $K_j\subset U_j$
is compact for $j=1,2$, so the hulls $\hat K_j$ are compact since the domains $U_j$ are holomorphically convex. It is easily seen that $\hat K\subset \hat K_1\times\hat K_2$, and since
the latter is compact it follows that $\hat K$ is compact, as required.
\end{proof}
Note this implies polydiscs $\pdisc(z,r)\subset\C^n$ are domains of holomorphy.

As with convex domains, the union of even two domains of holomorphy need not be a domain of holomorphy -- recall that $\pdisc(0,(e,e^2))\cup \pdisc(0,(e^2,e))\subset\C^2$
is a not a domain of holomorphy, but the polydiscs $\pdisc(0,(e,e^2))$ and $\pdisc(0,(e^2,e))$ are domains of holomorphy.
However, we will find that the union of an increasing sequence of domains of holomorphy is a domain of holomorphy:
\begin{theorem}[Behnke-Stein theorem]
\index{Behnke-Stein theorem}
\label{behnke}
Let $\{U_j\}_{j\geq 1}$ be a sequence of domains of holomorphy such that $U_j\subset U_{j+1}$ for all $j\geq 1$. Then $U:=\bigcup_{j\geq 1} U_j$ is a domain of holomorphy.
\end{theorem}
We delay the proof of this fact until Section~\ref{chap:levi}. 

The following is immediate from the equivalence of domains of holomorphy and holomorphically convex domains:
\begin{proposition}
\label{holinvbi}
Let $U\subset\C^n$ be a domain of holomorphy and $\phi\colon U\to\C^n$ a biholomorphic mapping. Then $\phi(U)$ is a domain of holomorphy. 
\end{proposition}

We conclude this section by introducing a particular type of open set of holomorphy which will be extremely useful later:
\begin{definition}
Let $U\subset\C^n$ be an open set and $f_1,\dots,f_m\in\hol(U)$. If $V:=\{z\in U\colon |f_j(z)|<1,1\leq j\leq m\}$ satisfies $\clos{V}\subset U$ then
$V$ is an \textbf{analytic polyhedron}. \index{analytic polyhedron}
If $K:=\{z\in U\colon |f_j(z)|\leq 1,\,1\leq j\leq m\}$ is compact then $K$ is a \textbf{compact analytic polyhedron}. In either case the set $\{f_1,\dots,f_m\}$
is a \textbf{frame} for the (compact) analytic polyhedron.\index{frame}
\end{definition}
Note that (compact) analytic polyhedra need not be connected.
\begin{proposition}
\label{apolhol}
If $V\subset\C^n$ is an analytic polyhedron then $V$ is an open set of holomorphy.
\end{proposition}
\begin{proof}
By definition there exists an open set $U\subset\C^n$ and functions $f_1,\dots,f_m\in\hol(U)$ such that $V=\{z\in U\colon |f_j(z)|<1,\,1\leq j\leq m\}$
and $\clos{V}\subset U$. Let $W$ be a component of $V$ and let $K\subset W$ be compact. We will show the holomorphically convex hull $\hat K:=\hat K_{\hol(W)}$ is compact, and it
is enough to show $\pmetric(\hat K,\bd W)>0$.
Suppose, for a contradiction, that this is not the case, so as in the proof of Theorem~\ref{usefuldh} 
there is a sequence $\{z_k\}_{k\geq 1}\subset \hat K$ with $z_k\to a\in\bd W$, and note $a\in\bd V$ because $W$
is a component of $V$.
Let $s:=\max_{1\leq j\leq m}\{\|f_j\|_K\}<1$.
If $z\in\hat K$ then $|g(z)|\leq \|g\|_K$ for all $g\in\hol(W)$ and in particular for $g:=f_j\big|_W$ ($1\leq j\leq m$), so $|f_j(z)|\leq \|f_j\|_K\leq s$.
Thus $|f_j(z_k)|\leq s$ for $1\leq j\leq m$ and $k\geq 1$, so $|f_j(a)|\leq s<1$ by continuity, meaning $a\in V$, which contradicts the fact that $a\in\bd V$. 
Therefore $\hat K$ is compact for all compact
$K\subset W$, so $W$ is a domain of holomorphy. This is true for each component $W$ of $V$, so $V$ is an open set of holomorphy.
\end{proof}
Next we have some approximation results which will be useful in Section~\ref{chap:levi}:
\begin{proposition}
\label{capolapprox}
Let $K\subset\C^n$ be a compact analytic polyhedron and let $V$ be a neighbourhood of $K$.
Then there exists an open set of holomorphy $X$ such that $K\subset X\subset V$.
\end{proposition}
\begin{proof}
By definition there exists an open set $U\subset\C^n$ and functions $f_1,\dots,f_m\in\hol(U)$ such that $K=\{z\in U\colon |f_j(z)|\leq 1,\,1\leq j\leq m\}$. Passing
to a subset if necessary we may assume $V$ is bounded and that $\clos{V}\subset U$, so there is a bounded neighbourhood $W$ of $\clos{V}$ such that $\clos{W}\subset U$.
 Therefore $\clos{W}\setminus V$ is compact, and because for each $z\in \clos{W}\setminus V$ we have $z\not\in K$ and thus
$|f_j(z)|>1$ for some $1\leq j\leq m$ (and this inequality holds in a neighbourhood of $z$), there exists $s>1$ such that whenever
$z\in \clos{W}\setminus V$ we have $|f_j(z)|\geq s$ for some $1\leq j\leq m$.
Therefore for $z\in W$, if $|f_j(z)|<s$ for each $1\leq j\leq m$ then $z\in V$. That is, the analytic polyhedron
$X:=\{z\in W\colon |f_j(z)/s|<1,\,1\leq j\leq m\}$ satisfies $K\subset X\subset V$ (note that $X$ is indeed an analytic polyhedron because $\clos{X}\subset \clos{V}\subset W$).
By Proposition~\ref{apolhol}, $X$ is an open set of holomorphy.
\end{proof}
\begin{proposition}
\label{approxcapol}
Let $U\subset\C^n$ be a domain, $K\subset U$ a compact set such that $K=\hat K_{\hol(U)}$ and $V\subset U$ a neighbourhood of $K$. Then there exists a compact analytic polyhedron
$L$ with frame in $H(U)$ such that $K\subset L\subset V$.
\end{proposition}
\begin{proof}
Passing to a subset of $V$ if necessary we may assume that $\bd V$ is compact and that $\clos{V}\subset U$.
Since $\hat K_{\hol(U)}=K$, for any $z\in U\setminus K$ there is a function $f\in\hol(U)$ such that $\|f\|_K<|f(z)|$ (and this inequality also holds for points in a neighbourhood of $z$),
and scaling $f$ if necessary we may assume $\|f\|_K\leq 1<|f(z)|$. By the compactness of 
$\bd V$ there are finitely many functions $f_1,\dots,f_m\in\hol(U)$ with $\|f_j\|_K\leq 1$ for
each $1\leq j\leq m$ and
such that if $z\in \bd V$ then $|f_j(z)|>1$ for some $1\leq j\leq m$. Let $L:=\{z\in V\colon |f_j(z)|\leq 1,1\leq j\leq m\}$, so clearly $K\subset L\subset V$.
Certainly $L$ is bounded and closed in the subspace topology on $V$, and $\emetric(L,\bd V)>0$ (otherwise there would be a sequence $\{z_j\}_{j\geq 1}\subset L$ with $z_j\to a\in\bd V$,
and by continuity we would have $|f_j(a)|\leq 1$ for $1\leq j\leq m$, which is a contradiction), so $L$ is compact. Therefore $L$ is the required compact analytic polyhedron.
\end{proof}
\section{Pseudoconvexity}
\label{chap:pseudo}
In the previous section the definition of a holomorphically convex domain was motivated by one of the several equivalent conditions for geometric convexity. In this section we generalise the other
characterisations of convexity given in Section~\ref{chap:convexity} to define various types of pseudoconvexity, and explore the consequences of these definitions.
Most importantly we demonstrate that the definitions of pseudoconvexity are in fact equivalent, and that every domain of holomorphy is pseudoconvex. The material for this section is synthesised from 
\cite{boas}, \cite{range}, \cite{itca} and \cite{vlad}.
\subsection{The continuity principle}
First we generalise the continuity principle introduced in Section~\ref{chap:convexity}. In that section we considered convergent sequences of line segments,
which can be viewed as the images of the interval $[-1,1]$ under affine maps. In Section~\ref{chap:hol} the definition of a convex hull was generalised by passing from affine functions to holomorphic
functions. We apply a similar method here and thus obtain our first type of pseudoconvex domain.

We first need some preliminary definitions:
\begin{definition}
A \textbf{holomorphic disc} is a continuous map $S\colon\clos{\ball(0,1)}\subset\C\to\C^n$ whose restriction to $\ball(0,1)$ is holomorphic.
The \textbf{boundary} of the holomorphic disc $S$, denoted $\bd S$, is $S(\bd\ball(0,1))$. \index{holomorphic disc}
\end{definition}
A holomorphic disc $S\colon\clos{\ball(0,1)}\to\C^n$ and the image $S(\ball(0,1))$ will often both be denoted by $S$. 
\begin{definition}
Let $U\subset\C^n$ be a domain. 
We say that $U$ \textbf{satisfies the continuity principle} if, for every sequence of holomorphic discs $\{S_j\}_{j\geq 1}$ satisfying $S_j\cup\bd S_j\subset U$ for all $j\geq 1$,
$\bd S_j\to B$ and $S_j\to T$ where $B\subset U$ and $T\subset \C^n$ are compact, we have $T\subset U$ (where we use the $L^\infty$ metric to define convergence of sets). \index{continuity principle}
\end{definition}
Note that with $B$ and $T$ defined as above we always have that $T$ is connected and $B\subset T$. This observation implies the following:
\begin{proposition}
\label{ctyint}
Let $U\subset\C^n$ and $V\subset\C^n$ be domains satisfying the continuity principle. Then each component of $U\cap V$ satisfies the continuity principle.
\end{proposition}

Next we can show that domains of holomorphy satisfy the continuity principle.
First we have a maximum modulus principle which follows directly from the maximum modulus principle of single variable complex analysis:
\begin{lemma}
Let $S\colon\clos{\ball(0,1)}\to\C^n$ be a holomorphic disc with $S\cup\bd S\subset U$ where $U\subset\C^n$ is a domain, and let $f\in\hol(U)$.
Then $\|f\|_S\leq \|f\|_{\bd S}$.
\end{lemma}
\begin{theorem}
\label{holcp}
If $U\subset\C^n$ is a domain of holomorphy then $U$ satisfies the continuity principle. 
\end{theorem}
\begin{proof}
We show the contrapositive, so suppose $U$ does not satisfy the continuity principle. Thus there is a sequence of holomorphic discs $\{S_j\}_{j\geq 1}$ satisfying $S_j\cup\bd S_j\subset U$
for all $j\geq 1$, $\bd S_j\to B$ and $S_j\to T$ where $B\subset U$ and $T\subset \C^n$ are compact, and $T\not\subset U$. Since $B\subset U$ is compact we have that
the $2r$-dilation $\dil{B}{2r}\subset U$ for some $r>0$. Let $K:=\clos{\dil{B}{r}}$, and note that $K\subset U$ is compact. We will show that 
$\hat K$ is not compact, which will show that $U$ is not holomorphically convex and hence not a domain of holomorphy. By convergence $\bd S_j\to B$ there exists
$J>0$ such that $j\geq J$ implies $\bd S_j\subset \dil{B}{r}\subset K$. By the above maximum modulus principle, for any
$j\geq J$ and $f\in\hol(U)$ we have $\|f\|_{S_j}\leq \|f\|_{\bd S_j}\leq \|f\|_K$,
meaning $S_j\subset\hat K$ for all $j\geq J$. Let $a\in T\setminus U$, so by convergence $S_j\to T$ there is a sequence $\{z_j\}_{j\geq J}$ with $z_j\to a$ and $z_j\in S_j$ for each $j$.
Thus $\{z_j\}_{j\geq J}\subset\hat K$ but $a\not\in\hat K$ (because $\hat K\subset U$ by definition), so $\hat K$ is not compact. Therefore $U$ is not holomorphically convex and is hence
not a domain of holomorphy.
\end{proof}
This result is useful for finding domains which are not domains of holomorphy. For instance: \begin{proposition}
Suppose $U\subset\C^n$ (with $n\geq 2$) is a domain such that $U^c$ is non-empty and compact. Then $U$ is not a domain of holomorphy.
\end{proposition}
\begin{proof}
Since $U^c$ is compact, the function $z\mapsto |z|$ takes its maximal value $r$ on $U^c$ at a point $a\in U^c$. We will assume that $a=(r,0,\dots,0)$ (if this is not the case we may apply a complex rotation
and appeal to the result of Proposition~\ref{holinvbi}).
For each $j\geq 1$ let $S_j\colon \clos{\ball(0,1)}\to U$ be given by $S_j(z):=(r+1/j,z,0,\dots,0)$, so clearly each $S_j$ is a holomorphic disc satisfying
$S_j\cup \bd S_j\subset U$. Furthermore, $S_j\to T$ where $T:=\{(r,z,0,\dots,0)\colon z\in\clos{\ball(0,1)}\}$ is compact
and $\bd S_j\to B$ where $B:=\{(r,z,0,\dots,0)\colon z\in\bd\ball(0,1)\}\subset U$ is also compact. But $T\not\subset U$ because $a\in T\cap U^c$,
so $U$ does not satisfy the continuity principle and hence is not a domain of holomorphy.
\end{proof}
We remark that in fact functions holomorphic on such domains extend to functions holomorphic on all of $\C^n$, but this does not follow from the above result and instead can be shown using
an integral representation formula such as that of Martinelli \cite{martinelli} and Bochner \cite{bochner} to explicitly realise the holomorphic extension (see \cite[page 172]{itca}
for the details).

\subsection{Plurisubharmonic functions}
Next we define pseudoconvexity by generalising Proposition~\ref{convexex}, which states that the convex domains are those which admit convex exhaustion functions. Recall that a convex function
is a continuous function with the property that for any two points in the graph of the function, the graph lies below the straight line between those points.
We will generalise this definition to functions of a complex variable to define (pluri)subharmonic functions, which satisfy the property that 
on each disc in the domain, the graph of the function 
lies below the graph of any harmonic function which dominates the function on the boundary of the disc. Here we formally define (pluri)subharmonic functions and give several important properties. 
We omit the proofs since they can be quite technical and may be found
in many introductory complex analysis textbooks (for example \cite[subsection 38]{itca} or \cite[section 10]{vlad}).

It will be convenient to introduce a weak notion of continuity for functions taking extended real values:
\begin{definition}
A function $f\colon S\subset\C^n\to [-\infty,\infty]$ is \textbf{upper-semicontinuous} at $s\in S$ if, for all $\alpha>f(s)$, there exists a neighbourhood $U$ of $s$ such that
$f(z)<\alpha$ for all $z\in U$. If $f$ is upper-semicontinuous at each point of $S$ then it is said to be \textbf{upper-semicontinuous} on $S$.
\index{upper-semicontinuous function}
\end{definition}
The condition of upper-semicontinuity is clearly weaker than continuity and describes functions which do not increase by more than arbitrarily small amounts in neighbourhoods of points,
but may decrease by any amount. 
We have a simple consequence of the definition which we will need later:
\begin{proposition}
Let $f\colon K\subset\C^n\to [-\infty,\infty)$ (where $K$ is compact) be upper-semicontinuous. Then $f$ is bounded above on $K$ and attains its maximum. \end{proposition}
Now we may define subharmonic functions:
\begin{definition}
Let $f\colon U\subset \C\to [-\infty,\infty)$ (where $U$ is open) be upper-semicontinuous. Then $f$ is \textbf{subharmonic} on $U$ if, for every $a\in U$ and $r>0$ such that $\clos{\ball(a,r)}\subset U$
 and for every continuous function $\phi\colon \clos{\ball(a,r)}\to\R$
which is harmonic on $\ball(a,r)$ and satisfies $\phi(z)\geq f(z)$ for all $z\in\bd\ball(a,r)$, we have $\phi(z)\geq f(z)$ for all $z\in\ball(a,r)$.
\index{subharmonic function}

Let $g\colon V\subset\C^n\to [-\infty,\infty)$ 
(where $V$ is open) be an upper-semicontinuous function. Then $f$ is \textbf{plurisubharmonic} on $V$ if, for every $a\in V$ and $\delta\in\C^n$ with $|\delta|=1$, the function
$z\mapsto f(a+\delta z)$ is subharmonic on $\{z\in\C\colon a+\delta z\in V\}$.\index{plurisubharmonic function}
\end{definition}
It may be verified that the sum of two plurisubharmonic functions is again plurisubharmonic, and that plurisubharmonicity is a local property,
which immediately implies
the following:
\begin{proposition}
\label{pshlocal}
Let $f\colon U\subset\C^n\to [-\infty,\infty)$ 
(where $U$ is open) be upper-semicontinuous. Suppose that, for every $a\in U$ and $\delta\in\C^n$ with $|\delta|=1$, the function $z\mapsto f(a+\delta z)$ is subharmonic 
on an open set $V$ with $0\in V\subset \{z\in\C\colon a+\delta z\in U\}$. Then $f$ is plurisubharmonic on $U$.
\end{proposition}
For twice-differentiable functions there is an equivalent condition for plurisubharmonicity that further emphasises the connection with convexity:
\begin{proposition}
\label{leviformpsh}
Let $U\subset\C^n$ be open and $f\colon U\to\R$ a $\cts^2$ function (that is, $f$ is $\cts^2$ when regarded as a function of $2n$ real variables). 
Then $f$ is plurisubharmonic if and only if for all $\delta\in\C^n$ and $z\in U$ we have
$$\Delta_\delta f(z):=\sum_{j,k=1}^n\frac{\partial^2 f}{\partial z_j\partial\con z_k}\bigg|_z\delta_j\con\delta_k \geq 0.$$
\end{proposition}
This result motivates the following definition which will be important later:
\begin{definition}
Let $f\colon U\subset \C^n\to\R$ (where $U$ is open) be a $\cts^2$ function. Then $f$ is \textbf{strictly plurisubharmonic} if $\Delta_\delta f(z)>0$ for all $z\in U$ and
$\delta\in\C^n\setminus\{0\}$.
\end{definition}
For example, the function $z\mapsto |z|^2$ is strictly plurisubharmonic, as $\Delta_\delta |z|^2 = |\delta|^2$. Proposition~\ref{leviformpsh} implies that strictly plurisubharmonic functions are plurisubharmonic.

Next we have three results which will be used
to show that a domain satisfies the continuity principle if and only if it is pseudoconvex:
\begin{proposition}
\label{pshhol}
Let $U\subset\C^n$ be open and let $g\in\hol(U)$ be non-vanishing. Then $f\colon U\to\R$ given by $f(z):=-\ln|g(z)|$ is plurisubharmonic on $U$.
\end{proposition}

\begin{proposition}
\label{pshupper}
Let $U\subset\C^n$ be open, and suppose that for each $\lambda\in\Lambda$ (where $\Lambda$ is some indexing set) there is a plurisubharmonic function $f_\lambda\colon U\to [-\infty,\infty)$, and that
the function $f$ defined by $f(z):=\sup_{\lambda\in\Lambda}\{f_\lambda(z)\}$ maps into $[-\infty,\infty)$ and is upper-semicontinuous. Then $f$ is plurisubharmonic.
\end{proposition}

\begin{proposition}
\label{pshmax}
Let $U\subset\C^n$ be a domain, $S\colon \clos{\ball(0,1)}\to U$ a holomorphic disc and $f\colon U\to [-\infty,\infty)$ a plurisubharmonic function. Then 
$\sup_{z\in S}\{f(z)\}\leq \sup_{z\in\bd S}\{f(z)\}$.
\end{proposition}
We conclude this subsection with an approximation result and its converse which will be vital in the following sections:
\begin{proposition}
\label{pshapprox}
Let $U\subset\C^n$ be a domain and $f\colon U\to[-\infty,\infty)$ a plurisubharmonic function. Then $f$ is the pointwise limit
of a non-increasing sequence of $\cts^\infty$ strictly plurisubharmonic functions $f_j\colon U_j\to\R$, where $\{U_j\}_{j\geq 1}$ is a sequence
of bounded domains satisfying $U_j\subset U_{j+1}$ for each $j\geq 1$ and
$\bigcup_{j\geq 1}U_j=U$.
\end{proposition}
\begin{proposition}
\label{decpsh}
Let $U\subset\C^n$ be open and $f_j\colon U\to[-\infty,\infty)$ a non-increasing sequence of plurisubharmonic functions which converges
pointwise to $f\colon U\to[-\infty,\infty)$. Then $f$ is plurisubharmonic.
\end{proposition}
\subsection{Global pseudoconvexity}
As discussed in the previous subsection, the plurisubharmonic functions of complex variables are a natural generalisation of the convex functions of real variables. Together with the characterisation
of convexity in terms of convex exhaustion functions (Proposition~\ref{convexex}), this suggests the following definition:
\begin{definition}
A domain $U\subset\C^n$ is \textbf{pseudoconvex} if it admits a continuous plurisubharmonic exhaustion function.\index{pseudoconvex domain}
\end{definition}
\begin{example}
\label{ballex}
Let $U:=\ball(a,r)$ for some $a\in\C^n$ and $r>0$. Define $f\colon U\to\R$ by $f(z):=-\ln\emetric(z,\bd U)$, so clearly $f$ is continuous and an exhaustion function for $U$ (it tends to $+\infty$ as the boundary
points are approached). One may verify that $f(z)=-\ln\inf_{w\in\bd U} \{r^{-1} |\!\sp{z-w}{a-w}\!|\} = \sup_{w\in\bd U} \{-\ln |r^{-1}\! \sp{z-w}{a-w}\!|\}$, and 
since each $z\mapsto r^{-1} \sp{z-w}{a-w}$ is holomorphic and non-vanishing on $U$ we see from Propositions \ref{pshhol} and \ref{pshupper} that $f$ is plurisubharmonic. Thus $U$ is pseudoconvex.
A similar argument shows that
polydiscs are pseudoconvex.
\end{example} 
\begin{proposition}
\label{psint}
Let $U\subset\C^n$ and $V\subset\C^n$ be pseudoconvex domains. Then each component of $U\cap V$ is pseudoconvex.
\end{proposition}
\begin{proof}
It is easily verified that the pointwise maximum of two continuous plurisubharmonic exhaustion functions for $U$ and $V$ yields such an exhaustion function for each component of $U\cap V$.
\end{proof}

We will prove that the domains satisfying the continuity principle are precisely the pseudoconvex domains and thus demonstrate the first part of the connection between domains
of holomorphy and pseudoconvexity. \begin{theorem}
\label{pscty}
If a domain $U\subset\C^n$ is pseudoconvex then it satisfies the continuity principle.
\end{theorem}
\begin{proof}
By pseudoconvexity, $U$ admits a continuous plurisubharmonic exhaustion function $f\colon U\to\R$. Let $\{S_j\}_{j\geq 1}$ be a sequence of holomorphic discs
satisfying $S_j\cup \bd S_j\subset U$ for all $j\geq 1$, $\bd S_j\to B$ and $S_j\to T$ where $B\subset U$ and $T\subset \C^n$ are compact. Suppose, for a contradiction, that $T\not\subset U$,
so there exists $a\in T\setminus U$, and by convergence $S_j\to T$ we have a sequence $\{z_j\}_{j\geq 1}\subset U$ with $z_j\in S_j$ for each $j\geq 1$ and $z_j\to a$. 
By convergence $\bd S_j\to B$ and compactness of $B$ there exists a compact $K\subset U$ and an integer $J\geq 1$ such that $\bd S_j\subset K$ for $j\geq J$. Let $M:=\|f\|_K<\infty$,
so for $j\geq J$ we have $f(z_j)\leq M$ (by Proposition~\ref{pshmax}).
Therefore $z_j\in f^{-1}((-\infty,M])$
for all $j\geq J$. But $z_j\to a$ and $f^{-1}((-\infty,M])$ is compact (since $f$ is an exhaustion function), so $a\in f^{-1}((-\infty,M])$, which contradicts the fact that $a\not\in U$.
Thus $T\subset U$, so $U$ satisfies the continuity principle.
\end{proof}
The converse is a consequence of the following intermediate results:
\begin{lemma}
\label{bdedispsh}
If $U\subset\C^n$ is a bounded non-empty domain which satisfies the continuity principle then $z\mapsto -\ln\emetric(z,\bd U)$ is plurisubharmonic (and hence $U$ is pseudoconvex).
\end{lemma}
\begin{proof}
First we note that $z\mapsto |z|^2-\ln\emetric(z,\bd U)$ is a continuous exhaustion function for $U$ (see Example~\ref{exex}) and that $z\mapsto |z|^2$ is plurisubharmonic, so
plurisubharmonicity of $z\mapsto -\ln\emetric(z,\bd U)$ implies pseudoconvexity of $U$. It remains to demonstrate plurisubharmonicity of $z\mapsto -\ln\emetric(z,\bd U)$.

For $\delta\in\C^n$ with $|\delta|=1$ and $z\in U$ define $\emetric_\delta(z):=\inf_{w\in\C}\{|w|\colon z+\delta w\in\bd U\}$, and notice that
$-\ln\emetric(z,\bd U)=-\ln\inf_{|\delta|=1} \{\emetric_\delta(z)\}=\sup_{|\delta|=1}\{-\ln \emetric_\delta(z)\}$, 
so by Proposition~\ref{pshupper} it is enough to show that
$\emetric_\delta$ is plurisubharmonic for each $\delta$. 

Let $\delta\in\C^n$ with $|\delta|=1$ be fixed and for brevity set $f:=\emetric_\delta$ and  
$L_z:=\{z+\delta w\colon w\in\C\}$ (for $z\in U$). It is easily verified that $f(z)=\emetric(z,\bd U\cap L_z)$ and that $f$ is upper-semicontinuous.
Let $a\in U$ and $\delta'\in\C^n$ with $|\delta'|=1$ be arbitrary. We will show that $z\mapsto f(a+\delta' z)$ is subharmonic in a neighbourhood of $0$, and it will follow from Proposition~\ref{pshlocal}
that $f$
is plurisubharmonic on $U$. 

Suppose, for a contradiction, that $z\mapsto f(a+\delta' z)$ is not subharmonic in a neighbourhood of $0$. Then there exists $R>0$
and a continuous function $\phi\colon\clos{\ball(0,R)}\subset\C\to\R$ which is harmonic on $\ball(0,R)$ and satisfies $\phi(z)\geq f(a+\delta' z)$ for $z\in\bd\ball(0,R)$
but $f(a+\delta' z)>\phi(z)$ at a point of $\ball(0,R)$. Clearly $z\mapsto f(a+\delta' z)-\phi(z)$ is upper-semicontinuous on $\clos{\ball(0,R)}$ so it attains
its maximum value $\epsilon>0$ at a point $z_0\in\ball(0,R)$. Set $\psi(z):=-\phi(z)-\epsilon$, so $\psi$ is continuous on $\clos{\ball(0,R)}$ and
harmonic on $\ball(0,R)$. For $z\in\bd\ball(0,R)$ we have $\psi(z)+f(a+\delta'z)<0$, so by upper-semicontinuity this inequality also holds for $z\in\bd\ball(0,r)$ when $r<R$
is sufficiently close to $R$. Thus there exists $r$ with $|z_0|<r<R$ such that whenever $z\in\bd\ball(0,r)$ we have $\psi(z)<-f(a+\delta'z)$.
It is also clear that
when $z\in\ball(0,r)$
we have $\psi(z)\leq -f(a+\delta' z)$, and that $\psi(z_0)=-f(a+\delta'z_0)$.

By definition 
$f(a+\delta' z_0)=-\ln \emetric_\delta(a+\delta'z_0)=-\ln\emetric(a+\delta' z_0,\bd U\cap L_{a+\delta'z_0})$,
and $\emetric(a+\delta'z_0,\bd U\cap L_{a+\delta'z_0})=|a+\delta'z_0-b|$ for some $b\in\bd U\cap L_{a+\delta'z_0}$,
so $b=a+\delta'z_0+\delta w_0$ for some $w_0\in\C$
and $f(a+\delta' z_0)=-\ln|w_0|$. 
Let $g\in\hol(\ball(0,R))$ be a holomorphic function with real part $\psi$.
Since $|w_0|=e^{-f(a+\delta'z_0)}=e^{\psi(z_0)}$, by adding a constant imaginary number to $g$ if necessary we may assume that $e^{g(z_0)}=w_0$. 

For $t\in (0,1]$ let
$S_t\colon\clos{\ball(0,1)}\to \C^n$ be given by $S_t(z):=a+\delta'rz+\delta t e^{g(rz)}$, so each $S_t$ is a holomorphic disc. Furthermore, as $j\to\infty$ we have $S_{1-1/j}\to S_1\cup \bd S_1$
and $\bd S_{1-1/j}\to \bd S_1$. Notice that $S_t(z)\in L_{a+\delta'rz}$ for all $t\in (0,1]$ and $z\in\clos{\ball(0,1)}$, so when $|te^{g(rz)}|<\emetric_\delta(a+\delta'rz)$
we have $S_t(z)\in U$. But $|te^{g(rz)}|=te^{\psi(rz)}$, when $z\in\clos{\ball(0,1)}$ we have $\psi(rz)\leq -f(a+\delta'rz)=\ln\emetric_\delta(a+\delta'rz)$,
and when $z\in\bd\ball(0,1)$ this inequality is strict.
Therefore when $t<1$ and $z\in\clos{\ball(0,1)}$ we have $|te^{g(rz)}|<e^{\psi(rz)}\leq \emetric_\delta(a+\delta'rz)$, which implies that $S_t\cup\bd S_t\subset U$ for $t<1$. Similarly,
when $t=1$ and $z\in\bd\ball(0,1)$ we have $|te^{g(rz)}|=e^{\psi(rz)}<\emetric_\delta(a+\delta'rz)$, meaning $\bd S_1\subset U$. 
That is, we have a sequence $\{S_{1-1/j}\}_{j\geq 1}$ of holomorphic discs satisfying $S_{1-1/j}\cup\bd S_{1-1/j}\subset U$ for each $j\geq 1$, 
$S_{1-1/j}\to S_1\cup\bd S_1$ and $\bd S_{1-1/j}\to\bd S_1$ where $S_1\cup \bd S_1$ and $\bd S_1\subset U$ are compact, so by the continuity principle we have $S_1\cup\bd S_1\subset U$. But
$S_1(z_0/r)=a+\delta'z_0+\delta e^{g(z_0)}=a+\delta'z_0+\delta w_0=b\in\bd U$, which is a contradiction. Therefore $z\mapsto f(a+\delta' z)$ is subharmonic in a neighbourhood of $0$, as required.
\end{proof}
\begin{lemma}
\label{incps2}
Let $\{U_j\}_{j\geq 1}$ be a sequence of pseudoconvex domains of $\C^n$ with $U_j\subset U_{j+1}$ for each $j\geq 1$. Then $U:=\bigcup_{j\geq 1} U_j$ is pseudoconvex.
\end{lemma}
\begin{proof}
If $U=\emptyset$ the result is trivial ($z\mapsto |z|^2$ is a continuous plurisubharmonic exhaustion function), 
so we assume this is not the case. By fixing $z_0\in U_1$ and replacing each $U_j$ by the component of $U_j\cap \ball(z_0,j)$ containing $z_0$ 
we may assume that each $U_j$ is bounded and non-empty (note that each such domain is pseudoconvex by Example~\ref{ballex} and Proposition~\ref{psint}).

As in the proof of Lemma~\ref{bdedispsh}, pseudoconvexity of $U$ will follow from plurisubharmonicity of $z\mapsto -\ln\emetric(z,\bd U)$.
Let $J\geq 1$ and $z\in U_J$. An elementary argument shows that the sequence $\{\emetric(z,\bd U_j)\}_{j\geq J}$ is non-decreasing and converges to $\emetric(z,\bd U)$, so by
continuity of $\ln$ we see that $\{-\ln\emetric(z,\bd U_j)\}_{j\geq J}$ is non-increasing
and converges to $-\ln\emetric(z,\bd U)$.
By Theorem~\ref{pscty} and Lemma~\ref{bdedispsh} we know that each $z\mapsto -\ln\emetric(z,\bd U_j)$ is plurisubharmonic, so from
Proposition~\ref{decpsh} it follows that $z\mapsto -\ln\emetric(z,\bd U)$ is plurisubharmonic
 on $U_J$. Since plurisubharmonicity is local and $U=\bigcup_{J\geq 1} U_J$ it follows
that $z\mapsto -\ln\emetric(z,\bd U)$ is plurisubharmonic on $U$.
\end{proof}
\begin{theorem}
\label{ctyps}
If a domain $U\subset\C^n$ satisfies the continuity principle then it is pseudoconvex.
\end{theorem}
\begin{proof}
If $U=\emptyset$ the result is trivial, so suppose $U\neq\emptyset$. Let $z_0\in U$ and for $j\geq 1$ let $U_j$ be the component of $U\cap\ball(z_0,j)$ containing $z_0$, so $U_j$ satisfies
the continuity principle by Proposition~\ref{ctyint}, $U_j\subset U_{j+1}$ for each $j$ and $U=\bigcup_{j\geq 1} U_j$. The result now follows from Lemmas~\ref{bdedispsh} and \ref{incps2}.
\end{proof}
Along with Theorem~\ref{holcp}, this result shows that domains of holomorphy are pseudoconvex. We will show later that the converse is also true and thus arrive at the solution to the Levi problem. 

Theorem~\ref{pscty}, together with the proofs of Lemma~\ref{bdedispsh} and Theorem~\ref{ctyps}, immediately implies the following useful fact:
\begin{corollary}
A domain $U\subset\C^n$ (with $\bd U\neq\emptyset$) is pseudoconvex if and only if $z\mapsto -\ln\emetric(z,\bd U)$ is plurisubharmonic.
\end{corollary}
\subsection{Local pseudoconvexity}
Recall that convexity of a domain is a local property of the boundary. We may introduce such a local notion of pseudoconvexity:
\begin{definition}
Let $U\subset\C^n$ be a domain and $a\in\bd U$. Then $U$ is \textbf{locally pseudoconvex} at $a$ if there is a neighbourhood $V$ of $a$
such that every component of $U\cap V$ is pseudoconvex. If $U$ is locally pseudoconvex at each $a\in\bd U$ it is \textbf{locally pseudoconvex}.\index{locally pseudoconvex domain}\index{pseudoconvex domain}
\end{definition}
We will find that pseudoconvexity and local pseudoconvexity are equivalent, thus indicating that pseudoconvexity itself is a local property:\begin{theorem}
\label{pslocal}
Let $U\subset\C^n$ be a domain. Then $U$ is pseudoconvex if and only if $U$ is locally pseudoconvex.
\end{theorem}
\begin{proof}
If $U=\emptyset$ or $U=\C^n$ the result is trivial, so suppose this is not the case.

If $U$ is pseudoconvex then for any $a\in\bd U$ each component of $U\cap\ball(a,1)$ is pseudoconvex by Corollary~\ref{psint}, so $U$ is locally pseudoconvex.

Now assume $U$ is locally pseudoconvex. 
First suppose $U$ is bounded. Let $a\in\bd U$ and let $V\subset\C^n$ be a neighbourhood of $a$ such that each component of $U\cap V$ is pseudoconvex. Let
$r_a>0$ so that $\ball(a,r_a)\subset V$, so each component of $U\cap V\cap \ball(a,r_a)=U\cap\ball(a,r_a)$ is pseudoconvex (by Corollary~\ref{psint}).
Consider such a component $W$, so the function $z\mapsto -\ln\emetric(z,\bd W)$ is plurisubharmonic on $W$. 
Clearly $\emetric(z,\bd W)=\min(\emetric(z,\bd U),\emetric(z,\bd\ball(a,r_a)))$,
so when $z\in\ball(a,r_a/2)\cap W$ we have $\emetric(z,\bd W)=\emetric(z,\bd U)$ by virtue of the fact that $\emetric(z,\bd U)\leq \emetric(z,a)<r_a/2<\emetric(z,\bd\ball(a,r_a))$. Thus
$z\mapsto -\ln\emetric(z,\bd U)$ is plurisubharmonic on $\ball(a,r_a/2)\cap W$. This is true for each component $W$ and each boundary point $a$, so $z\mapsto -\ln\emetric(z,\bd U)$ 
is plurisubharmonic on $X:=U\cap\bigcup_{a\in\bd U}\ball(a,r_a/2)$. Let $Y:=U\cap\bigcup_{a\in\bd U}\ball(a,r_a/4)$, so since $Y\subset U$ is open and $U$ is bounded, $K:=U\setminus Y$ 
is compact and hence $z\mapsto -\ln\emetric(z,\bd U)$ attains a maximum value $M$ on $K$. Define $f\colon U\to\R$ by $f(z):=\max(-\ln\emetric(z,\bd U),M)$, 
so $f$ is continuous, and it is plurisubharmonic on $X$ by Proposition~\ref{pshupper} and plurisubharmonic on the interior of $K$ because it is constant there.
Clearly $f$ is also an exhaustion function for $U$ (it tends to $+\infty$ as boundary points are approached). Therefore $U$ is pseudoconvex.

If $U$ is unbounded, fix $z_0\in U$ and for each $j\geq 1$ let $U_j$ be the connected component of $U\cap\ball(z_0,j)$ containing $z_0$. We have $U=\bigcup_{j\geq 1}U_j$ and $U_j\subset U_{j+1}$
for all $j\geq 1$, so by Lemma~\ref{incps2} it suffices to show each $U_j$ is pseudoconvex, and by the above argument it is enough
to show each $U_j$ is locally pseudoconvex. Let $a\in\bd U_j$, so either $a\in\bd U$ or $a\in\bd\ball(z_0,j)$.
If $a\in\bd U$ there is a neighbourhood $V$ such that each component of $U\cap V$ is pseudoconvex, so by Corollary~\ref{psint}
each component of $U\cap V\cap \ball(z_0,j)=U_j\cap V$ is pseudoconvex. Now suppose $a\not\in\bd U$, so $a\in\bd\ball(z_0,j)$. Let $r:=\emetric(a,\bd U)>0$,
$V:=\ball(a,r/2)$ and $W$ a component of $U_j\cap V$. Then for $z\in W$ we have $\emetric(z,\bd W)=\min(\emetric(z,\bd V),\emetric(z,\bd U_j))$, and 
$\emetric(z,\bd U_j)=\emetric(z,\bd\ball(z_0,j))$ (because $\emetric(z,\bd U_j)=\min(\emetric(z,\bd U),\emetric(z,\bd\ball(z_0,j)))$, and when $z\in W$
we have $\emetric(z,\bd\ball(z_0,j))<r/2<\emetric(z,\bd U)$). Therefore $-\ln\emetric(z,\bd W)=\max(-\ln\emetric(z,\bd V),-\ln\emetric(z,\bd \ball(z_0,j)))$ for $z\in W$,
and the functions $z\mapsto -\ln\emetric(z,\bd V)$ and $z\mapsto -\ln\emetric(z,\bd\ball(z_0,j))$ are plurisubharmonic on $W$ by pseudoconvexity of $V=\ball(a,r/2)$ and $\ball(z_0,j)$,
so $z\mapsto -\ln\emetric(z,\bd W)$ is plurisubharmonic by Proposition~\ref{pshupper}. Thus $W$ is pseudoconvex, so each component of $ U_j\cap V$ is pseudoconvex. Therefore
each $U_j$ is bounded and locally pseudoconvex and is hence pseudoconvex, so it follows
that $U$ is pseudoconvex. 
\end{proof}
\subsection{Levi pseudoconvexity}
As with convex domains, domains with twice-differentiable boundaries admit an additional local characterisation of pseudoconvexity. \begin{definition}
Let $U\subset\C^n$ be a domain and $a\in\bd U$. If $V\subset\C^n$ is a neighbourhood of $a$ and $f\colon V\to\R$ is such that $U\cap V=f^{-1}((-\infty,0))$ then
$f$ is a \textbf{local defining function} for $U$ at $a$. If $\clos{U}\subset V$ then $f$ is a \textbf{global defining function} for $U$.
\index{local defining function}\index{global defining function}

Let $k\geq 1$ (we allow $k=\infty$). The domain $U$ is said to have \textbf{$k$-times differentiable boundary} or \textbf{$\cts^k$ boundary} at a boundary point if there is a $\cts^k$
 local defining function whose gradient is non-zero at the point. If a domain has $\cts^k$ boundary at each boundary point it is said to have \textbf{$k$-times differentiable
boundary} or \textbf{$\cts^k$ boundary}.
If $a\in\bd U$ and $f\colon V\to\R$ is a $\cts^k$ local defining function for $U$ at $a$ then the \textbf{(complex) tangent space} to $\bd U$ at $a$ with respect to $f$ is
$ T^c_a(f):=\{\delta\in\C^n\colon\sprod{\delta}{\con{\cgrad{f}(a)}}=0\}$ (where $\cgrad:=(\partial/\partial z_1,\dots,\partial/\partial z_n)$ is the complex gradient).\index{differentiable boundary}\index{tangent space}
\end{definition}
Note that a global defining function immediately induces local defining functions at each boundary point.
We remark that when working with a domain known to have $k$-times differentiable boundary at some point, we require that \emph{all} defining functions be $\cts^k$ and have non-zero gradient at
the point.
In view of equation~\eqref{defnais} of Remark~\ref{defrem} (which applies in this complex setting by simply regarding subsets of $\C^n$ as subsets of $\R^{2n}$), by passing to subsets if necessary
we will always assume local defining functions $f\colon V\to\R$ for domains $U$ with $\cts^k$ boundaries satisfy the property that $\bd U\cap V=f^{-1}(\{0\})$ (which implies
$V\setminus\clos{U}=f^{-1}((0,\infty))$).

With the results of Section~\ref{chap:convexity} in mind we give the following definitions:
\begin{definition}
Let $U\subset\C^n$ be a domain with $\cts^2$ boundary. Then $U$ is \textbf{Levi pseudoconvex} at $a\in\bd U$ if 
there is a $\cts^2$ local defining function $f\colon V\to\R$ (where $V$ is a neighbourhood of $a$)
such that $\Delta_\delta f(a)\geq 0$ for all $\delta\in T^c_a(f)$. 
If this inequality is strict for all $\delta\in T^c_a(f)\setminus\{0\}$
then $U$ is \textbf{strictly pseudoconvex} at $a$. If $U$ is Levi pseudoconvex at each boundary point it is said to be
\textbf{Levi pseudoconvex}. If $U$ is bounded and strictly pseudoconvex at each boundary point it is said to be \textbf{strictly pseudoconvex}.\index{Levi pseudoconvex domain}\index{strictly pseudoconvex domain}
\end{definition}
First note that a strictly pseudoconvex domain is Levi pseudoconvex.
Also note the connection with plurisubharmonicity -- if at each boundary point of a domain there 
is a non-degenerate $\cts^2$ plurisubharmonic (respectively, strictly plurisubharmonic) local defining function then the domain is Levi pseudoconvex (respectively, strictly pseudoconvex).
\begin{definition}
Let $U\subset\C^n$ be a bounded domain that admits a $\cts^\infty$ strictly plurisubharmonic global defining function which is non-degenerate on $\bd U$.
We say that $U$ is a \textbf{smoothly bounded strictly pseudoconvex} domain.\index{smoothly bounded strictly pseudoconvex domain}
\end{definition}
It is clear that smoothly bounded strictly pseudoconvex domains have $\cts^\infty$ boundaries, and in fact the converse also holds -- strictly pseudoconvex domains with $\cts^\infty$
boundaries are smoothly bounded in the sense of the above definition (see \cite[page 59]{range}). 
Regardless of this fact, we will always show explicitly that a $\cts^\infty$ strictly plurisubharmonic global defining function exists when verifying that a strictly pseudoconvex
domain is smoothly bounded,
rather than simply showing that the boundary is $\cts^\infty$.

It may be verified through a technical argument that non-negativity (and positivity) of the quantity $\Delta_\delta f(a)$ is independent of the chosen local defining function (see \cite[page 56]{range}),
which means that the Levi or strict pseudoconvexity of a domain at a boundary point can be determined by considering a single defining function
at that point (rather than all possible defining functions).
\begin{example}
Let $U:=\ball(c,r)\subset\C^n$ and let $f\colon\C^n\to\R$ be given by $f(z):=|z-c|^2-r^2$, so $f$ is a global defining function for $U$. We have
$ \frac{\partial f}{\partial z_j}\big|_z=\con z_j-\con c_j$, so $\cgrad{f}(z)=\con{z}-\con{c}$, and because this is non-zero for $z\in\bd U$ it follows that $U$ has $\cts^2$ boundary.
Furthermore, $\frac{\partial^2 f}{\partial z_j\partial\con z_k}\big|_z$ is $1$ if $j=k$ and $0$ otherwise. Therefore, for any
$a\in\bd U$ and $\delta\in\C^n\setminus\{0\}$, $\Delta_\delta f(a)=|\delta|^2>0$. It follows that $U$ is strictly pseudoconvex.
\end{example}
\begin{example}
Let $U:=\pdisc(c,r)\subset\C^n$ ($n\geq 2$) where $r$ is a positive vector radius, and let $a\in\bd U$ such that $|a_1-c_1|=r_1$ and $|a_j-c_j|<r_j$ for $j=2,\dots,n$. Then
for a neighbourhood $V$ of $a$ the function $f\colon V\to\R$ given by $f(z):=|z_1-c_1|^2-r_1^2$ is a local defining function for $U$. Clearly $f$ is $\cts^2$ and $\cgrad{f}(a)=(\con a_1-\con c_1,0,\dots,0)\neq 0$,
so $U$ has $\cts^2$ boundary at $a$. But $\frac{\partial^2 f}{\partial z_j\partial\con z_k}\big|_a$ is $1$ if $j=k=1$ and $0$ otherwise, so $\Delta_\delta f(a)=|\delta_1|^2$ for all $\delta\in\C^n$.
But if $\delta\in T^c_a(f)$ then $\delta_1=0$ and thus $\Delta_\delta f(a)=0$. Hence $U$ is Levi pseudoconvex at $a$ (and it is also pseudoconvex by earlier results) but not strictly pseudoconvex.
\end{example}
Thus there are pseudoconvex domains which are not strictly pseudoconvex. However, we have the following approximation result:
\begin{proposition}
\label{slpapprox}
Let $U\subset\C^n$ be a pseudoconvex domain. Then there exists a sequence $\{U_j\}_{j\geq 1}$ of smoothly bounded 
strictly pseudoconvex domains 
such that $\clos{U_j}\subset U_{j+1}$ for all $j\geq 1$
and $U=\bigcup_{j\geq 1}U_j$.
\end{proposition}
To prove this we will require a technical lemma of real analysis which depends on a theorem of Morse (see \cite{morse} for a proof):
\begin{theorem}[Morse's theorem]
Let $f\colon U\subset\R^n\to\R$ (where $U$ is open) be $\cts^\infty$. Then the Lebesgue measure of $\{f(x)\colon x\in U,\grad{f}(x)=0\}$ is zero.
\end{theorem}
This immediately implies the following:
\begin{lemma}
\label{morselemma}
Let $f\colon U\subset\C^n\to\R$ be $\cts^\infty$ and $V\subset\R$ a non-empty open set. Then there exists $y\in V$ such that $\cgrad{f}(z)\neq 0$ for all $z\in f^{-1}(\{y\})$.
\end{lemma}

We need one more lemma, which shows that $\cts^1$ functions yield defining functions for components of their sublevel sets. 
\begin{lemma}
\label{globdefworks}
Let $f\colon U\subset\C^n\to\R$ be $\cts^1$, let $\alpha\in f(U)$ such that $\cgrad{f}(z)\neq 0$ for all $z\in f^{-1}(\{\alpha\})$ and let $V$ be a component of $f^{-1}((-\infty,\alpha))$.
Then $z\mapsto f(z)-\alpha$ is a global defining function for $V$ on some neighbourhood of $\clos V$.
\end{lemma}
We omit the proof (which proceeds via the implicit function theorem
as described in Remark~\ref{defrem}) since it is a simple argument of real analysis and not especially instructive.
\begin{proof}[Proof of Proposition~\ref{slpapprox}]
If $U=\emptyset$ the assertion is trivial, so assume this is not the case.
Let $f\colon U\to\R$ be a continuous plurisubharmonic exhaustion function for $U$, so by Proposition~\ref{pshapprox} there is a sequence of bounded domains $\{V_j\}_{j\geq 1}$
with $V_j\subset V_{j+1}$ for all $j\geq 1$ and $U=\bigcup_{j\geq 1}V_j$, and a non-increasing sequence of $\cts^\infty$ strictly plurisubharmonic functions $f_j\colon V_j\to \R$ which converges pointwise 
to $f$.
Let $z_0\in V_1$, and suppose without loss of generality that $f_1(z_0)\leq 0$ (if this is not the case subtract $f_1(z_0)$ from $f$ and each $f_j$). Note 
$[0,\infty)\subset f(U)$ by the intermediate value theorem, since $f(z_0)\leq f_1(z_0)\leq 0$ and $f$ tends to $+\infty$ as boundary points are approached.
For each $j\geq 1$ the non-empty set $K_j:=f^{-1}((-\infty,j+1/2])$ is compact, so there is some $k_j\geq 1$ with $K_j\subset V_{k_j}$. By passing to a subsequence of $\{V_j\}_{j\geq 1}$
if necessary we may assume $K_j\subset V_j$ for all $j\geq 1$. For each $j\geq 1$ let $r_j\in (j-1/2,j+1/2)$ be such that $\cgrad{f_j}(z)\neq 0$ whenever $f_j(z)=r_j$ (the possibility of
this is assured by Lemma~\ref{morselemma}), and note $r_j\in f_j(V_j)$ by the intermediate value theorem, since $f_j(z_0)\leq 0<r_j$ and for some $z\in K_j\subset V_j$ we have $f(z)=j+1/2$ and thus
$f_j(z)\geq j+1/2>r_j$.

For each $j\geq 1$ set $U_j:=f_j^{-1}((-\infty,r_j))\subset K_j$. If $z\in \clos{U_j}$ for some $j\geq 1$ then $f_j(z)\leq r_j$, so $f_{j+1}(z)\leq f_j(z)\leq r_j<r_{j+1}$ and thus $z\in U_{j+1}$, so $\clos{U_j}\subset U_{j+1}$.
If $z\in U$ there exists $K\geq 1$ such that $r_K>f(z)$, so by convergence $f_j(z)\to f(z)$ there exists $J\geq 1$ such that $j\geq J$ implies $f_j(z)<r_K$, and in particular if we take $j\geq \max\{J,K\}$
we have $f_j(z)<r_K\leq r_j$, so $z\in U_j$. Therefore $U=\bigcup_{j\geq 1} U_j$. Now replace each $U_j$ with the component of $U_j$ containing $z_0$, so the properties
$\clos{U_j}\subset U_{j+1}$ for $j\geq 1$ and $U=\bigcup_{j\geq 1} U_j$ are retained, and now
$z\mapsto f_j(z)-r_j$ is a global defining function for each $U_j$ by Lemma~\ref{globdefworks}, meaning each $U_j$ is a smoothly bounded strictly pseudoconvex domain.
\end{proof}
Next we show that pseudoconvex domains with $\cts^2$ boundary are Levi pseudoconvex.
We will need a fact from geometry on smoothness of the so-called \emph{signed distance function}, which is a particularly convenient defining function:
\begin{proposition}
\label{sigsmooth}
Let $U\subset\C^n$ be a domain with $\cts^2$ boundary and define $\eta\colon\C^n\to\R$ by
$$ \eta(z) := \begin{cases} -\emetric(z,\bd U),& z\in\clos{U} \\ \emetric(z,\bd U),& z\in U^c.\end{cases}$$
Then $\eta$ is a local defining function for $U$ at each boundary point (in particular $\eta$ is $C^2$ in a neighbourhood of $\bd U$ and is non-degenerate on $\bd U$).
\end{proposition}
This is a consequence of the implicit function theorem, but the details are rather technical and not particularly relevant so we omit the proof and refer the interested 
reader to \cite[page 136]{krantz}.
We will also need a lemma:
\begin{lemma}
\label{levilemma}
Suppose $f\colon U\to\R$ is $\cts^2$ and has non-vanishing gradient, and that for some bounded set $S$ with $\clos{S}\subset U$ we have
$\Delta_\delta f(z)\geq 0$ for all $z\in S$ and $\delta\in\C^n$ with $\sprod{\delta}{\con{\cgrad{f}(z)}}=0$. Then there exists $c\geq 0$ such that 
$\Delta_\delta f(z)\geq -c|\delta|\left|\sprod{\delta}{\con{\cgrad{f}(z)}}\right|$ for all $z\in S$ and $\delta\in\C^n$.
\end{lemma}
\begin{proof}
If $S=\emptyset$ the assertion is trivial, so assume this is not the case.
Define
$$ d:=2\sum_{j,k=1}^n \left\|\frac{\partial^2f}{\partial z_j\partial\con z_k}\right\|_S<\infty, \qquad c:=2d\left\|\frac{1}{\left|\cgrad{f}\right|}\right\|_S<\infty.$$
Let $z\in S$ and $\delta\in\C^n$ be arbitrary. We will show the required inequality is satisfied for $c$ as above, which does not depend on $z$ or $\delta$.
Set $\nu:=\cgrad{f}(z)/\left|\cgrad{f}(z)\right|$. Let $\delta':=\sprod{\delta}{\con\nu}\con\nu$ and $\delta'':=\delta-\delta'$, so $\sprod{\delta''}{\con\nu}=0$.
Therefore, since $\Delta_{\delta''}f(z)\geq 0$, 
\begin{align*}
 \Delta_\delta f(z)=\Delta_{\delta''} f(z)+\Delta_{\delta'} f(z)+\sum_{j,k=1}^n \frac{\partial^2f}{\partial z_j\partial\con z_k}\bigg|_z \left(\delta''_j\con{\delta'}_k+\delta'_j\con{\delta''}_k\right) 
	\geq 0+s+t,
\end{align*}
where $s$ is the second term of the middle expression and $t$ is the third. Clearly $|s|/|\delta'|^2$ and $|t|/|\delta'\delta''|$ are at most $d$,
and since $|\delta'|,|\delta''|\leq |\delta|$ it follows
that $|s|\leq d|\delta'||\delta|$ and $|t|\leq d|\delta'||\delta|$. This implies that 
\begin{align*}
\Delta_\delta f(z)\geq -2d|\delta'||\delta|=-2d|\delta|\frac{1}{\left|\cgrad{f}(z)\right|}\left|\sprod{\delta}{\con{\cgrad{f}(z)}}\right| \geq -c|\delta|\left|\sprod{\delta}{\con{\cgrad{f}(z)}}\right|,
\end{align*}
as required.
\end{proof}
\begin{theorem}
\label{pslevi}
Let $U\subset\C^n$ be a pseudoconvex domain with $\cts^2$ boundary. Then $U$ is Levi pseudoconvex.
\end{theorem}
\begin{proof}
Let $\eta\colon\C^n\to\R$ be the signed distance function.
By pseudoconvexity of $U$
the function $f(z):=-\ln (-\eta(z))=-\ln\emetric(z,\bd U)$ is plurisubharmonic on $U$. Thus for $z\in U$ and $\delta\in\C^n$ with $\sprod{\delta}{\con{\cgrad{\eta}(z)}}=0$ we have, by the chain rule,
\begin{align*}
 0 \leq \Delta_\delta f(z)&=\sum_{j,k=1}^n \left(-\frac{\partial^2\eta}{\partial z_j\partial\con z_k}\bigg|_z \frac{\delta_j\con\delta_k}{\eta(z)}+\frac{\partial \eta}{\partial z_j}\bigg|_z\frac{\partial \eta}{\partial\con z_k}\bigg|_z\frac{\delta_j\con\delta_k}{\eta(z)^2}\right) =-\frac{1}{\eta(z)}\Delta_\delta \eta(z),\end{align*}
and since $\eta(z)<0$ it follows that $\Delta_\delta\eta(z)\geq 0$. Let $a\in\bd U$, and let $V$ be a neighbourhood of $a$ on which $\eta$ is $\cts^2$ and has non-vanishing gradient (the possibility of this
is assured by Proposition~\ref{sigsmooth} together with continuity of the gradient), and let $W\subset U\cap V$ be a bounded open set with $a\in\clos{W}\subset V$. By Lemma~\ref{levilemma} there exists $c\geq 0$
such that $\Delta_\delta\eta(z)\geq -c|\delta|\left|\sprod{\delta}{\con{\cgrad{\eta}(z)}}\right|$ for all $z\in W$ and $\delta\in\C^n$. Fix $\delta\in\C^n$, so 
because the first two derivatives of $\eta$ are continuous on $V$ this inequality holds as $z\to a$, meaning $\Delta_\delta\eta(a)\geq -c|\delta|\left|\sprod{\delta}{\con{\cgrad{\eta}(a)}}\right|$.
This holds for all $\delta\in\C^n$, so if $\delta\in T^c_a(\eta)$ then $\Delta_\delta\eta(a)\geq 0$ (as $\sprod{\delta}{\con{\cgrad{\eta}(a)}}=0$). This is true for all $a\in\bd U$, so
$U$ is Levi pseudoconvex.
\end{proof}
We conclude the section by demonstrating the converse, thus proving that for domains with twice-differentiable boundaries the concepts of Levi pseudoconvexity and pseudoconvexity are equivalent.
\begin{theorem}
\label{leviisps}
Let $U\subset\C^n$ be Levi pseudoconvex. Then $U$ is pseudoconvex.
\end{theorem}
\begin{proof}
By Theorem~\ref{pslocal} it is enough to show $U$ is locally pseudoconvex at each boundary point. 
Let $a\in\bd U$ and let $f\colon V\to\R$ be a local defining function for $U$ at $a$. 
Suppose, without loss of generality, that $\pd{f}{x_{2n}}\big|_a>0$ (regarding $f$ as a function of real variables), so as in Remark~\ref{defrem} 
there are open sets $W\subset\R^{2n-1}$ and $I\subset\R$ such that $a\in W\times I\subset V$ and a $\cts^2$ function $g\colon W\to I$ such that
$f(x)=0$ if and only if $x_{2n}=g(x_1,\dots,x_{2n-1})$. Moreover, $h\colon W\times I\subset\C^n\to\R$ defined by $h(z):=\Im(z_n)-g(\Re(z_1),\dots,\Re(z_n),\Im(z_1),\dots,\Im(z_{n-1}))$ 
is a local defining function for $U$ at each point $b\in\bd U\cap(W\times I)$.
In particular, since $U$ is Levi pseudoconvex,
for each such $b$ we have $\Delta_\delta h(b)\geq 0$ for all $\delta\in T^c_b(h)$. Let $W'\times I'$ be a neighbourhood of $a$ such that $g(W')\subset I'$,
$\clos{W'}\subset W$ and $\clos{I'}\subset I$,
so by Lemma~\ref{levilemma} there exists $c\geq 0$ such that $\Delta_{\delta}h(b)\geq -c|\delta|\left|\sprod{\delta}{\con{\cgrad{h}(b)}}\right|$ for all $b\in \bd U\cap (W'\times I')$ and
$\delta\in\C^n$. But the derivatives of $h$ are independent of $\Im(z_n)$, and any point $z\in W'\times I'\subset\C^n$ can be written in the form $z=b+(0,\dots,0,i\alpha)$ for some
$b\in \bd U\cap(W'\times I')$ and $\alpha\in\R$, 
so this inequality holds not only on $\bd U$ but at every point of $W'\times I'$. Thus for $z\in W'\times I'$ and $\delta\in\C^n$ with $|\delta|=1$
we have
\begin{equation} \Delta_\delta h(z)\geq -c\left|\sprod{\delta}{\con{\cgrad{h}(z)}}\right|\geq -2c\left|\sprod{\delta}{\con{\cgrad{h}(z)}}\right|.\label{levipseq}\end{equation}

Let $r>0$ so that $\ball(a,r)\subset W'\times I'$, and
let $\mu\colon\ball(a,r)\cap U\to\R$
be given by $\mu(z):=-\ln(-h(z))+c^2|z|^2$, so clearly $\mu$ is continuous and $\mu(z)\to+\infty$ as $z\to b$ for all $b\in \ball(a,r)\cap \bd U$. We will show that $\mu$ is plurisubharmonic on 
$\ball(a,r)\cap U$, and
since $z\mapsto -\ln\emetric(z,\bd\ball(a,r))$ is plurisubharmonic by pseudoconvexity of $\ball(a,r)$ it will follow that $z\mapsto \mu(z)-\ln\emetric(z,\bd\ball(a,r))$ is a plurisubharmonic
exhaustion function for each component of 
$\ball(a,r)\cap U$ (since it tends to $+\infty$ as the boundary points are approached), which will show that $U$ is locally pseudoconvex at $a$. By the chain rule we have, for any $\delta\in\C^n$
with $|\delta|=1$ and $z\in\ball(a,r)\cap U$,
$$
 \Delta_\delta \mu(z)	=-\frac{1}{h(z)}\Delta_\delta h(z)+\frac{1}{h(z)^2}\left|\sprod{\delta}{\con{\cgrad{h}(z)}}\right|^2+c^2 \nonumber 
	\geq -2c\left|\frac{\sprod{\delta}{\con{\cgrad{h}(z)}}}{h(z)}\right|+\left|\frac{\sprod{\delta}{\con{\cgrad{h}(z)}}}{h(z)}\right|^2+c^2 \geq 0
$$
where for the first $\geq $ we have used \eqref{levipseq} and the fact that $|h(z)|=-h(z)$ since $h(z)<0$, and for the second we have observed that the left hand side is a perfect square.
Therefore $\mu$ is plurisubharmonic on $\ball(a,r)\cap U$, so as described above it follows that $U$ is locally pseudoconvex at $a$. This is true for each $a\in\bd U$, meaning $U$ is pseudoconvex.
\end{proof}

\section{The Levi problem}
\label{chap:levi}
We have now demonstrated that the domains of holomorphy and the holomorphically convex domains are identical, that all of these domains are pseudoconvex, and that the pseudoconvex domains
are precisely those which are locally pseudoconvex, satisfy the continuity principle, and, in the case of domains with twice-differentiable boundaries, are Levi pseudoconvex.
In this section we demonstrate that pseudoconvex domains are necessarily domains of holomorphy, thus completing the solution to the Levi problem. We apply methods based on those from
\cite{bers}, \cite{boas}, \cite{krantz}, \cite{range}
and \cite{itca}.
\subsection{The inhomogeneous Cauchy-Riemann equations}
It is an elementary fact that for a real-differentiable function $f\colon U\subset\C^n\to\C$ to be complex-differentiable at a point $a\in U$ it is necessary and sufficient that the 
partial derivatives $\pd{f}{\con z_j}\big|_a$ be $0$ for $1\leq j\leq n$.
We may reinterpret this condition in terms of differential forms.

Recall that a real differential form is a smooth alternating covariant tensor field. That is, a real $m$-form on an open $U\subset\R^n$ is a smooth map $\tau\colon U\times(\R^n)^m\to\R$ such that, for each
fixed $p\in U$, the map $\tau(p,\cdot)\colon(\R^n)^m\to\R$ is linear in each argument and alternating. There are important $1$-forms $dx_j$ (for $1\leq j\leq n$) defined by $dx_j(p,y_1,\dots, y_n):=y_j$.
An $m$-form $\tau$ and a $p$-form $\sigma$ can be combined via the wedge product to form an $(m+p)$-form $\tau\wedge\sigma$, and we recall that $\wedge$ is associative and that 
$\tau\wedge\sigma = (-1)^{mp}\sigma\wedge\tau$. 
We may define complex forms in terms of these real forms:
\begin{definition}
Let $U\subset\C^n$ be a domain and $m\geq 0$ an integer. A \textbf{smooth (complex) $m$-form} on $U$ is a map $\omega\colon U\times (\C^n)^m\to\C$ 
given by $\omega = \tau+i\sigma$, where $\tau$ and $\sigma$ are smooth $m$-forms on $U$ when it is regarded as a subset of $\R^{2n}$. The space of such smooth complex $m$-forms is denoted
$\form{m}(U)$. For $1\leq j\leq n$ we write
$dz_j:=dx_j+idx_{n+j}\in\form{1}(U)$ and $d\con z_j:=dx_j-idx_{n+j}\in\form{1}(U)$. Let $p\geq 0$. The \textbf{wedge product} $\wedge\colon\form{m}(U)\times\form{p}(U)\to\form{m+p}(U)$ extends
the usual wedge product according to
$$ (\tau_1+i\sigma_1)\wedge (\tau_2+i\sigma_2):= (\tau_1\wedge\tau_2-\sigma_1\wedge\sigma_2)+i(\tau_1\wedge \sigma_2+\sigma_1\wedge\tau_2)$$
for any $\tau_1+i\sigma_1\in\form{m}(U)$, $\tau_2+i\sigma_2\in\form{p}(U)$. A \textbf{smooth $(0,m)$-form} on $U$ is a smooth complex $m$-form $\omega$
which can be written in the form
\begin{equation} \omega = \sum_{j_1=1,\dots,j_m=1}^n \omega_{j_1,\dots,j_m} d\con z_{j_1}\wedge\dots \wedge d\con z_{j_m},\qquad \omega_{j_1,\dots,j_m}\in\ccts^\infty(U)\label{formform}\end{equation}
(where $\ccts^\infty(U)$ denotes the set of functions $U\to\C$ which are $\cts^\infty$ when regarded as functions of $2n$ real variables mapping into $\R^2$).
\end{definition}
We may introduce an exterior derivative:
\begin{definition}
Let $U\subset\C^n$ be a domain and $f\in\ccts^\infty(U)$ a function. Define the smooth $(0,1)$-form $\con\partial f$ by $$ \con\partial f:=\sum_{j=1}^n \frac{\partial f}{\partial \con z_j} d\con z_j. $$
For a smooth $(0,m)$-form $\omega$ given by \eqref{formform}, define the smooth $(0,m+1)$-form $\con\partial\omega$ by $$ \con\partial\omega := \sum_{j_1=1,\dots,j_m=1}^n (\con\partial \omega_{j_1,\dots,j_m})\wedge d\con z_{j_1}\wedge\dots\wedge d\con z_{j_m}.$$\index{exterior derivative}
\end{definition}
We note the connection with the Cauchy-Riemann conditions -- a real-differentiable function $f\colon U\subset\C^n\to\C$ is holomorphic on $U$ if and only if $\con\partial f=0$.

In some instances we will see that it is important to find a function $\phi\in\ccts^\infty(U)$ such that $\con\partial\phi=\omega$, where $\omega$ is a prescribed smooth $(0,1)$-form. Roughly speaking,
this will allow us to adjust smooth functions in order to make them holomorphic. If $\con\partial\phi=\omega$ then applying $\con\partial$ to both sides we have $\con\partial^2\phi=\con\partial\omega$,
and it is readily verified that $\con\partial^2=0$, so a necessary condition for the equation $\con\partial\phi=\omega$ to be solvable (for $\phi$) is that $\con\partial\omega=0$.
In 1965 H\"ormander proved, using Hilbert space theory,
that when $U$ is a pseudoconvex domain this condition is also sufficient \cite{hormander}. For our purposes a weaker result suffices, where we assume the domain is smoothly bounded and strictly pseudoconvex:
\begin{theorem}[H\"ormander's theorem]
\label{hormander}
\index{H\"ormander's theorem}
Let $U\subset\C^n$ be a smoothly bounded strictly pseudoconvex domain and let $\omega$ be a smooth $(0,1)$-form on $U$ with $\con\partial \omega=0$. Then there exists
$\phi\in\ccts^\infty(U)$ with $\con\partial \phi=\omega$.
\end{theorem}
We refer the reader to \cite[chapter 4]{krantz} for a proof.
\subsection{The Oka-Weil approximation theorem}
We will require the Oka-Weil approximation theorem, which is an important theorem of complex analysis in its own right:
\begin{theorem}[Oka-Weil theorem]
\label{okaweil}
\index{Oka-Weil theorem}
Let $U\subset\C^n$ be a domain of holomorphy and $K\subset U$ a compact subset with $K=\hat K$ (where $\hat K:=\hat K_{\hol(U)}$ is the holomorphically convex hull of $K$). Then for all $f\in \hol(K)$ and
$\epsilon>0$ there exists $g\in \hol(U)$ such that $\|f-g\|_K<\epsilon$.
\end{theorem}
Recall that for a function to be holomorphic on an arbitrary set $C$ (not necessarily open) it must be holomorphic on a neighbourhood of $C$.
Before we discuss the proof of this theorem we will introduce some useful terminology:
\begin{definition}
Let $A\subset B\subset\C^n$ be sets such that for any $f\in\hol(A)$, any compact $K\subset A$ and any $\epsilon>0$ there exists $g\in\hol(B)$ such that $\|f-g\|_K<\epsilon$. Then
we say $A$ \textbf{is Runge in} $B$ or $(A,B)$ is a \textbf{Runge pair}.\index{Runge pair}
\end{definition}
Note that if $A$ is compact, for $A$ to be Runge in $B$ it is necessary and sufficient that for all $f\in\hol(A)$ and $\epsilon>0$ there exists $g\in\hol(B)$ such that
$\|f-g\|_A<\epsilon$.
Thus the Oka-Weil approximation theorem states that if $K$ is a compact subset of a domain of holomorphy $U$ which satisfies $K=\hat K_{\hol(U)}$ then $(K,U)$ is a Runge pair. We have the following result which encapsulates the typical application of Runge pairs:
\begin{proposition}
\label{rungepair}
Let $\{K_j\}_{j\geq 1}$ be a sequence of compact sets in $\C^n$ such that $K_j\subset \inter K_{j+1}$ and $K_j$ is Runge in $K_{j+1}$ for all $j\geq 1$. 
Then $K_1$ is Runge in $U:=\bigcup_{j\geq 1}\inter K_j$.
\end{proposition}
\begin{proof}
Let $f\in\hol(K_1)$ be a function and $\epsilon>0$. Let $g_1:=f\in\hol(K_1)$. Since $K_1$ is Runge in $K_2$ there is
$g_2\in\hol(K_2)$ such that $\|g_2-g_1\|_{K_1}<\epsilon/2$. Similarly, $K_2$ is Runge in $K_3$ so there is $g_3\in\hol(K_3)$ such that $\|g_3-g_2\|_{K_2}<\epsilon/4$. In this way
we obtain, for all $j\geq 1$, functions $g_j\in\hol(K_j)$ such that $\|g_{j+1}-g_{j}\|_{K_j}<\epsilon/2^j$.

Now consider the series $g(z):=g_1(z)+\sum_{j=1}^\infty (g_{j+1}(z)-g_j(z))$. 
For any $J\geq 1$ we may write 
$ g(z)=
g_J(z)+\sum_{j=J}^\infty (g_{j+1}(z)-g_j(z)),$
so the value of $g(z)$ depends only on the tail of the sequence $\{g_j(z)\}_{j\geq 1}$. 
In particular, since for all $z\in U$ there exists $J\geq 1$ such that $j\geq J$ implies $z\in K_j$, $g$ is well-defined on all of $U$ (provided the series converges). 
Let $J\geq 1$. 
For any $j\geq J$ we have $\|g_{j+1}-g_j\|_{K_J}\leq \|g_{j+1}-g_j\|_{K_j}<\epsilon/2^j$, so the series converges uniformly on $K_J$ by the Weierstrass $M$-test.
It follows that $g$ is holomorphic on $\inter K_J$. This is true for each $J\geq 1$, so $g\in\hol(U)$.
Moreover, for any $z\in K_1$ we have
$$ |f(z)-g(z)|=|g(z)-g_1(z)|=\left|\sum_{j=1}^\infty (g_{j+1}(z)-g_j(z))\right|< \sum_{j=1}^\infty \frac{\epsilon}{2^j}=\epsilon,$$
which implies $\|f-g\|_{K_1}<\epsilon$ as required.\end{proof}

Before we prove the Oka-Weil theorem we need some intermediate facts. For the following three results and their proofs let $D:=\clos{\ball(0,1)}\subset\C$.
\begin{lemma}
\label{spsapproxcapol}
Let $K\subset\C^n$ be a compact analytic polyhedron and $U$ a neighbourhood of $K\times D$. Then there is an open set $X$ consisting of a union of disjoint smoothly bounded strictly pseudoconvex domains such that $K\times D\subset X\subset U$.
\end{lemma}
\begin{proof}
Let $\alpha>1$ be a real number and $W\subset\C^n$ a neighbourhood of $K$ such that $W\times\ball(0,\alpha)\subset U$ (this is possible by compactness of $K$ and $D$).
By Proposition~\ref{capolapprox} there is an open set of holomorphy $V$ such that $K\subset V\subset W$. Each component of $V$ is a domain of holomorphy,
so by Proposition~\ref{prodhol} each component of $V\times \ball(0,\alpha)$ is a domain of holomorphy and hence a pseudoconvex domain. Consider such a component $Y$, so $Y\cap (K\times D)$
is compact (as $K\times D$ does not intersect the boundary of $Y$). By Proposition~\ref{slpapprox},
$Y$ is given by an increasing union of smoothly bounded strictly pseudoconvex domains,
so in particular there is a domain $X_Y$ of this union such that $Y\cap (K\times D)\subset X_Y\subset Y$. 
Repeating this procedure for each component $Y$ of $V\times\ball(0,\alpha)$
and taking the union $X$ 
of the strictly pseudoconvex domains $X_Y$ yields the desired open set $X\subset V\times\ball(0,\alpha)\subset U$ (note the domains $X_Y$ are disjoint because the components $Y$ of $V\times\ball(0,\alpha)$
are, by definition, disjoint).
\end{proof}
\begin{lemma}
\label{sneakyapprox}
Let $K\subset\C^n$ be compact and $G\in\hol(K\times D)$. Then for $(z,w)$ in a neighbourhood of $K\times D$, 
$G(z,w)=\sum_{j=1}^\infty a_j(z)w^j$ where for $j\geq 1$ we have $a_j\in\hol(K)$. Moreover, the series converges uniformly on $K\times D$.
\end{lemma}
\begin{proof}
We have $G\in\hol(K_o\times D_o)$ where $K_o$ and $D_o:=\ball(0,\alpha)\subset\C$ (where $\alpha>1$) are neighbourhoods of $K$ and $D$ respectively.
For each $j\geq 1$ let $a_j(z):=\frac{1}{j!}\frac{\partial^j G}{\partial w^j}\big|_{(z,0)}$, so because the partial derivatives of a holomorphic function
are holomorphic it follows that $a_j\in\hol(K_o)\subset \hol(K)$. For each
fixed $z\in K_o$ the function $w\mapsto G(z,w)$ is holomorphic on $D_o$ (and by reducing $\alpha$ if necessary we may assume $w\mapsto G(z,w)$ is continuous
on $\clos{D_o}$), so $G(z,w)=\sum_{j=0}^\infty \frac{1}{j!}\frac{\partial^j G}{\partial w^j}\big|_{(z,0)}w^j=\sum_{j=0}^\infty a_j(z)w^j$ for all $z\in K_o$ and $w\in D_o$.

Next we show uniform convergence on $K\times D$. For fixed $z\in K$ the function $w\mapsto G(z,w)$ is continuous on $\clos{D_o}$ and
holomorphic on $D_o$, so by the Cauchy estimate 
$$|a_j(z)|= \left|\frac{1}{j!}\frac{\partial^j G}{\partial w^j}\bigg|_{(z,0)}\right|\leq \frac{\|G\|_{\{z\}\times \clos{D_o}}}{\alpha^j} \leq \frac{\|G\|_{K\times\clos{D_o}}}{\alpha^j}.$$
Thus on $K\times D$ the terms of the series for $G$ are bounded in norm by $\|G\|_{K\times\clos{D_o}}(w/\alpha)^j$, which are terms of a convergent geometric series (because $|w|\leq 1<\alpha$).
By the Weierstrass $M$-test it follows that the series for $G$ converges uniformly on $K\times D$.
\end{proof}
\begin{lemma}
\label{capolrunge1}
Let $U\subset\C^n$ be a domain, $K\subset U$ a compact analytic polyhedron, $f\in\hol(U)$ and $L:=\{z\in K\colon |f(z)|\leq 1\}$. Then $L$ is Runge in $K$.
\end{lemma}
\begin{proof}
Let $g\in\hol(L)$ and $\epsilon>0$. We are to find $h\in\hol(K)$ such that $\|g-h\|_L<\epsilon$.

Since $g\in\hol(L)$ we have $g\in\hol(L_o)$ for some neighbourhood $L_o\subset U$ of $L$. 
Choose a bounded open $V$ such that $L\subset V\subset \clos V\subset L_o$. If $K\setminus V=\emptyset$ then $K\subset V\subset L_o$, meaning $g\in\hol(K)$, so if we set $h:=g$ then $h\in\hol(K)$
and $\|g-h\|_L=0<\epsilon$, as required. For the rest of the proof assume that $K\setminus V\neq\emptyset$.
Let $\chi\colon \C^n\to\R$ be $\cts^\infty$ and compactly supported in $L_o$ with $\chi(z)=1$ for $z\in V$. We have $\min_{z\in K\setminus V}\{|f(z)|\}>1$,
so there is a neighbourhood $W\subset U$ of $K\setminus V$ and a number $\alpha>1$ such that $|f(z)|>\alpha$ for all $z\in W$.
Let $\omega(z,w):=\frac{g(z)\con\partial\chi(z)}{f(z)-w}$ for $z\in W\cup V$ and $w\in\ball(0,\alpha)\subset\C$. If $f(z)=w$ then $z\not\in W$ and thus $z\in V$, so $\con\partial\chi=0$ in a neighbourhood
of $z$. Therefore $\omega$ is a smooth $(0,1)$-form, and because $f$ and $g$ are holomorphic we have $\con\partial\omega=0$ on $(W\cup V)\times \ball(0,\alpha)$.  

Clearly $(W\cup V)\times\ball(0,\alpha)$ 
is a neighbourhood of $K\times D$, so by Lemma~\ref{spsapproxcapol} there is an open set $X$ given by a union of disjoint smoothly bounded strictly pseudoconvex domains 
such that $K\times D\subset X\subset (W\cup V)\times\ball(0,\alpha)$.
Thus we may apply
H\"ormander's theorem to the restriction of $\omega$ to each component of $X$, and this will yield a function $\phi\in\ccts^\infty(X)$ such that
$\con\partial\phi = \omega$, so $g(z)\con\partial\chi(z)=(f(z)-w)\con\partial\phi(z,w)$. 

Let $G\colon X\to\C$ be given by $G(z,w):=g(z)\chi(z)-(f(z)-w)\phi(z,w)$, so $G$ is $\ccts^\infty$ and we see that
$\con\partial G(z,w)=g(z)\con\partial\chi(z)-(f(z)-w)\con\partial\phi(z,w)=0$, so $G\in\hol(X)$. 
By Lemma~\ref{sneakyapprox} there are functions $\{a_j\}_{j\geq 1}\subset\hol(K)$ such that $G(z,w)=\sum_{j=1}^\infty a_j(z)w^j$ in a neighbourhood of $K\times D$, with uniform convergence on $K\times D$.
For $z\in L$ we have $f(z)\in D$, and thus $g(z)=G(z,f(z))=\sum_{j=1}^\infty a_j(z)f(z)^j$ with uniform convergence on $L$.
Therefore for some large $m\geq 1$ the function $h(z):=\sum_{j=1}^m a_j(z)f(z)^j\in\hol(K)$ satisfies $\|g-h\|_L<\epsilon$, as required.

It follows that $L$ is Runge in $K$.
\end{proof}
This lemma admits the following easy generalisation:
\begin{corollary}
\label{capolrunge}
Let $U\subset\C^n$ be a domain, and let $K\subset U$ and $L\subset K$ be compact analytic polyhedra such that the frame of $L$ is in $\hol(U)$.
Then $L$ is Runge in $K$.
\end{corollary}
\begin{proof}
There exist functions $f_1,\dots,f_m\in\hol(U)$ such that $L=\{z\in K\colon |f_j(z)|\leq 1,\,1\leq j\leq m\}$.
Let $L_0:=K$, $L_1:=\{z\in L_0\colon |f_1(z)|\leq 1\},\dots,L_m:=\{z\in L_{m-1}\colon |f_m(z)|\leq 1\}$. Note that each $L_j$ (for $0\leq j\leq m$) is a compact analytic polyhedron, and $L_m=L$.
Now let $g\in\hol(L)=\hol(L_m)$ and $\epsilon>0$. By Lemma~\ref{capolrunge1} there exists $h_{m-1}\in\hol(L_{m-1})$ such that $\|g-h_{m-1}\|_L<\epsilon/m$.
By Lemma~\ref{capolrunge1} again there exists $h_{m-2}\in\hol(L_{m-2})$ such that $\|h_{m-1}-h_{m-2}\|_{L_{m-1}}<\epsilon/m$ and thus $\|g-h_{m-2}\|_L<2\epsilon/m$. Repeating this argument
we obtain functions $h_j\in\hol(L_j)$ such that $\|g-h_j\|_L<(m-j)\epsilon/m$ for $0\leq j< m$. That is, there exists $h:=h_0\in\hol(L_0)=\hol(K)$ such
that $\|g-h\|_L<\epsilon$, so $L$ is Runge in $K$.
\end{proof}
This is the fundamental approximation result on which the proof of the Oka-Weil theorem is based, as it yields the following:
\begin{proposition}
\label{okaclose}
Let $U\subset\C^n$ be a domain of holomorphy and $K\subset U$ a compact analytic polyhedron with frame in $\hol(U)$. Then $K$ is Runge in $U$.
\end{proposition}
\begin{proof}
Let $\{K_j\}_{j\geq 1}$ be a sequence of compact sets with $K\subset \inter K_j\subset K_j\subset \inter K_{j+1}$ for all $j\geq 1$ and $\bigcup_{j\geq 1} \inter K_j=U$ (take, for instance,
$K_j:=\{z\in U\colon \emetric(z,\bd U)\geq r/j \text{ and } |z|\leq R+j\}$ for each $j\geq 1$, where $r:=\emetric(K,\bd U)/2>0$ and $R:=\max_{z\in K}\{|z|\}$). For each $j\geq 1$ the hull $\hat K_j$ is compact
(because $U$ is holomorphically convex) and satisfies $\hat {\hat K}_j=\hat K_j$, so by Proposition~\ref{approxcapol} there exist
compact analytic polyhedra $L_j$ with frames in $\hol(U)$ such that $\hat K_j\subset L_j\subset U$ for each $j\geq 1$. Clearly $U=\bigcup_{j\geq 1}\inter L_j$ (since $\inter K_j\subset \inter L_j\subset U$ for
all $j\geq 1$), so by passing
to a subsequence if necessary we may assume $L_j\subset \inter L_{j+1}$ for each $j\geq 1$. Let $L_0:=K$, and consider the sequence $\{L_j\}_{j\geq 0}$. For each
$j\geq 0$ we have $L_j\subset \inter L_{j+1}$, where $L_j$ and $L_{j+1}$ are compact analytic polyhedra with frames in $\hol(U)$, so 
$L_j$ is Runge in $L_{j+1}$ by Corollary~\ref{capolrunge}. But we also have $U=\bigcup_{j\geq 0} \inter L_j$, so by Proposition~\ref{rungepair} it follows that
$L_0=K$ is Runge in $U$.
\end{proof}
The Oka-Weil theorem is now easily proved:
\begin{proof}[Proof of Theorem~\ref{okaweil}]
We have a domain of holomorphy $U\subset\C^n$ and a compact subset $K\subset U$ with $K=\hat K$, and we are to show $K$ is Runge in $U$.
Let $f\in\hol(K)$ and $\epsilon>0$. It is enough to find $g\in\hol(U)$ so that $\|f-g\|_K<\epsilon$. Since $f\in\hol(K)$ there is a neighbourhood
$K_o$ of $K$ such that $f\in\hol(K_o)$. By Proposition~\ref{approxcapol} there is a compact analytic polyhedron $L$ with frame in $\hol(U)$ such that $K\subset L\subset K_o$.
By Proposition~\ref{okaclose} we know that $L$ is Runge in $U$, so since $f\in\hol(L)$ there is a function $g\in\hol(U)$ such that $\|f-g\|_K<\epsilon$, as required.
\end{proof}
We have a useful corollary:
\begin{corollary}
\label{okaweilcor}
Let $U\subset\C^n$ be a domain of holomorphy, $K\subset U$ a compact subset and $V\subset U$ a neighbourhood of $K$ such that $\hat K\cap \bd V=\emptyset$. Then
$\hat K\subset V$. \end{corollary}
\begin{proof}
Let $W:=\clos{V}^c$ be the exterior of $V$, so $\hat K\subset V\cup W$.
Let $f\in\hol(V\cup W)$ be identically equal to $0$ on $V$ and to $1$ on $W$, so by the Oka-Weil theorem (and the
fact that $\hat{\hat K}=\hat K$) there exists $g\in\hol(U)$ with $\|f-g\|_{\hat K}<1/2$. We have $f\equiv 0$ on $K$, so $\|g\|_K< 1/2$.
If $z\in W\cap\hat K$ then $f(z)=1$, so $|g(z)|>1/2> \|g\|_K$ and hence $z\not\in\hat K$, which is a contradiction. Thus $W\cap\hat K=\emptyset$, so $\hat K\subset V$.
\end{proof}
\subsection{The Behnke-Stein theorem}
Recall from Section~\ref{chap:hol} the Behnke-Stein theorem:
\newtheorem*{behnke}{Theorem~\ref{behnke}}
\begin{behnke}[Behnke-Stein theorem]
\index{Behnke-Stein theorem}
Let $\{U_j\}_{j\geq 1}$ be a sequence of domains of holomorphy such that $U_j\subset U_{j+1}$ for all $j\geq 1$. Then $U:=\bigcup_{j\geq 1} U_j$ is a domain of holomorphy.
\end{behnke}
To prove the Behnke-Stein theorem we will invoke the following lemmas (note that if $U=\emptyset$ or $U=\C^n$ the assertions of the results in this subsection
are trivial, so in the proofs we will assume this is not the case):
\begin{lemma}
\label{bslemma}
Let $\{K_j\}_{j\geq 1}$ be a sequence of compact sets in $\C^n$ such that $K_j\subset K_{j+1}$ and $K_j$ is Runge in $K_{j+1}$ for all $j\geq 1$. Let $U:=\bigcup_{j\geq 1} \inter K_j$, and suppose
there is a sequence $\{U_j\}_{j\geq 1}$ of domains of holomorphy such that $K_j\subset U_j\subset U_{j+1}\subset U$ for each $j\geq 1$. Then $U$ is a domain of holomorphy.
\end{lemma}
\begin{proof}
Using the fact that $U=\bigcup_{j\geq 1} \inter K_j$ we may pass to subsequences if necessary and assume, in addition to the hypotheses, that $K_j\subset\inter K_{j+1}$ for all $j\geq 1$.
By Proposition~\ref{rungepair} it follows that $K_j$ is Runge in $U$ for all $j\geq 1$.

Note that $U=\bigcup_{j\geq 1} U_j$, so $U$ is connected and hence a domain. We will show that $U$ is holomorphically convex. Let $K\subset U$ be compact, and let $a\in U$ with $\pmetric(a,\bd U)<\pmetric(K,\bd U)$.
We can show $a\not\in\hat K$. Let $j\geq 1$ be sufficiently large that $\{a\}\cup K\subset K_j$, and let $k>j$ be sufficiently large that
$\pmetric(a,\bd U_k)<\pmetric(K,\bd U_k)$ (this is possible because as $k\to\infty$ we have $\pmetric(a,\bd U_k)\to\pmetric(a,\bd U)$
and $\pmetric(K,\bd U_k)\to\pmetric(K,\bd U)$). By Lemma~\ref{hulldist} we have $a\not\in\hat K_{\hol(U_k)}$, so there is $f\in\hol(U_k)\subset \hol(K_j)$ such that
$\epsilon:=|f(a)|-\|f\|_K>0$. Since $K_j$ is Runge in $U$ there is $g\in\hol(U)$ such that $\|f-g\|_{K_j}<\epsilon/2$. Therefore $|g(a)|-\|g\|_K>|f(a)|-\|f\|_K-2\epsilon/2=0$,
so $a\not\in\hat K$. Therefore $\pmetric(\hat K,\bd U)\geq \pmetric(K,\bd U)>0$, so $\hat K$ is compact. This is true for each compact $K\subset U$, so $U$ is holomorphically convex and
hence a domain of holomorphy.
\end{proof}
We have another lemma, which will allow us to consider only bounded domains in the proof of the Behnke-Stein theorem:
\begin{lemma}
\label{berslemma}
Let $U\subset\C^n$ be a domain, let $z_0\in U$, and suppose the component of $U\cap\pdisc(z_0,r)$ containing $z_0$ is a domain of holomorphy for all $r>0$. Then $U$ is a domain of holomorphy. 
\end{lemma}
\begin{proof}
For each $j\geq 1$ let $U_j$ be the component of $U\cap\pdisc(z_0,j)$ containing $z_0$. We first make an observation. Suppose $K\subset U_j$ is compact, so by holomorphic convexity
of $U_{j+1}$ the hull $\hat K_{\hol(U_{j+1})}$ is compact. We also know
$\hat K_{\hol(U_{j+1})}\subset \hat K_{\hol(\pdisc(z_0,j+1))}$ since $U_{j+1}\subset\pdisc(z_0,j+1)$. Furthermore, 
$\pmetric(\hat K_{\hol(\pdisc(z_0,j+1))},\bd\pdisc(z_0,j+1))=\pmetric(K,\bd\pdisc(z_0,j+1))>1$ (we have used Lemma~\ref{hulldist} for the equality)
which implies $\hat K_{\hol(\pdisc(z_0,j+1))}\subset\pdisc(z_0,j)$. Therefore $\hat K_{\hol(U_{j+1})}\subset \pdisc(z_0,j)$, and obviously $\hat K_{\hol(U_{j+1})}\subset U_{j+1}\subset U$,
meaning $\hat K_{\hol(U_{j+1})}\subset U\cap\pdisc(z_0,j)$. It follows that $\hat K_{\hol(U_{j+1})}\cap \bd U_j=\emptyset$ (since $U_j$ is a component of $U\cap\pdisc(z_0,j)$), so by
Corollary~\ref{okaweilcor} we have $\hat K_{\hol(U_{j+1})}\subset U_j$. That is, if $K\subset U_j$ is compact then $\hat K_{\hol(U_{j+1})}$ is a compact subset of $U_j$.

Let $\alpha:=\pmetric(z_0,\bd U)/2$. 
For each $j\geq 1$ let $L_j\subset U_j$ 
be the component of $\{z\in U\colon \pmetric(z,\bd U)\geq \alpha/j\text{ and }\pmetric(z,z_0)\leq j-1\}$
containing $z_0$, so $L_j$ is compact, and let $K_j:=(\hat L_j)_{\hol(U_{j+1})}\subset U_j$ (where we have used the observation at the start of the proof).
Clearly $U=\bigcup_{j\geq 1}\inter L_j$ and $\inter L_j\subset \inter K_j\subset U$ for all $j\geq 1$, so $U=\bigcup_{j\geq 1}\inter K_j$. 
We also have $K_j\subset K_{j+1}$ and $K_j\subset U_j\subset U_{j+1}\subset U$ for all $j\geq 1$, and that each $U_j$ is a domain of holomorphy,
so by Lemma~\ref{bslemma} to show that $U$ is a domain of holomorphy it is enough to show that
$K_j$ is Runge in $K_{j+1}$ for each $j$. But for each $j\geq 1$, $K_j$ coincides with its $\hol(U_{j+1})$-convex hull, so $K_j$ is Runge in $U_{j+1}$ by the Oka-Weil theorem,
and so obviously $K_j$ is Runge in $K_{j+1}\subset U_{j+1}$.
Therefore $U$ is a domain of holomorphy.
\end{proof}
We need one more lemma:
\begin{lemma}
\label{vlemma}
Let $U\subset\C^n$ be a domain of holomorphy and let $z_0\in U$. Then for all $\epsilon>0$ with $\epsilon<\pmetric(z_0,\bd U)$, the component of $\{z\in U\colon \pmetric(z,\bd U)>\epsilon\}$
containing $z_0$ is a domain of holomorphy.
\end{lemma}
\begin{proof}
Suppose $0<\epsilon<\pmetric(z_0,\bd U)$.
Let the component of $\{z\in U\colon \pmetric(z,\bd U)>\epsilon\}$ containing $z_0$ be $V$. Let $K\subset V$ be compact, so by holomorphic convexity of $U$ and Lemma~\ref{hulldist}
the hull $\hat K_{\hol(U)}$ satisfies $\pmetric(\hat K_{\hol(U)},\bd U)=\pmetric(K,\bd U)>\epsilon$, meaning $\hat K_{\hol(U)}\subset \{z\in U\colon\pmetric(z,\bd U)>\epsilon\}$. 
It follows that $\hat K_{\hol(U)}\cap \bd V\neq\emptyset$,
so by Corollary~\ref{okaweilcor} we see that $\hat K_{\hol(U)}$ is a compact subset of $V$,
and obviously $\hat K_{\hol(V)}\subset\hat K_{\hol(U)}$, so $\hat K_{\hol(V)}$ is compact. This is true for each $K\subset V$, so $V$ is holomorphically convex and hence a domain of holomorphy.
\end{proof}
Finally we can prove the Behnke-Stein theorem:
\begin{proof}[Proof of Theorem~\ref{behnke}]
We have a sequence $\{U_j\}_{j\geq 1}$ of domains of holomorphy in $\C^n$ such that $U_j\subset U_{j+1}$ for all $j\geq 1$. We are to show $U:=\bigcup_{j\geq 1} U_j$ is a domain of holomorphy.
First suppose $U$ is bounded. Choose $z_0\in U_1$, set $\epsilon:=\pmetric(z_0,\bd U_1)/2$ and replace each $U_j$ with the component of $\{z\in U_j\colon \pmetric(z,\bd U_j)>\epsilon/j\}$ containing
$z_0$, so we still have that each $U_j$ is a domain of holomorphy (by Lemma~\ref{vlemma}) and that $U=\bigcup_{j\geq 1}U_j$, but now $\clos{U_j}\subset U_{j+1}$ for all $j\geq 1$ and in particular
$\pmetric(U_j,\bd U)>0$.

Now we construct a subsequence $\{V_k\}_{k\geq 1}$ of $\{U_j\}_{j\geq 1}$ satisfying, for all $k\geq 2$, the inequality
\begin{equation} \max_{z\in\bd V_k} \pmetric(z,\bd V_{k+1})<\pmetric(V_{k-1},\bd V_{k+1}).\label{bsineq}\end{equation}
Let $V_1:=U_1$, choose $j_2>1$ sufficiently large that $\{z\in U\colon\pmetric(z,\bd U)\geq \pmetric(V_1,\bd U)\}\subset U_{j_2}$ and let $V_2:=U_{j_2}$,
so $\max_{z\in\bd V_2}\pmetric(z,\bd U)<\pmetric(V_1,\bd U)$. As $j\to\infty$ we have $\pmetric(z,\bd U_j)\to\pmetric(z,\bd U)$ (for every $z\in \bd V_2$) 
and $\pmetric(V_1,\bd U_j)\to\pmetric(V_1,\bd U)$, so since $\bd V_2$ is compact it follows from basic metric space theory that 
there exists $J>j_2$ such that $j\geq J$ implies $\max_{z\in\bd V_2}\pmetric(z,\bd U_j)<\pmetric(V_1,\bd U_j)$. 
Now, as with our choice of $j_2$, choose $j_3\geq J$ sufficiently large that $\{z\in U\colon\pmetric(z,\bd U)\geq \pmetric(V_2,\bd U)\}\subset U_{j_3}$ and let $V_3:=U_{j_3}$,
so \eqref{bsineq} is satisfied for $k=2$ and $\max_{z\in\bd V_3}\pmetric(z,\bd U)<\pmetric(V_2,\bd U)$.

Now we proceed inductively.
Let $m\geq 3$, and suppose we have sets $\{V_k\}_{1\leq k\leq m}$ (with $V_k=U_{j_k}$ for each $k$) such that \eqref{bsineq} is satisfied for $2\leq k\leq m-1$ and 
$\max_{z\in\bd V_m}\pmetric(z,\bd U)<\pmetric(V_{m-1},\bd U)$ (the above argument yields such a construction in the base case $m=3$). 
As above, there exists $J>j_m$ such that $j\geq J$ implies $\max_{z\in\bd V_m}\pmetric(z,\bd U_j)<\pmetric(V_{m-1},\bd U_j)$. Again, 
as above we may choose $j_{m+1}\geq J$ so that $\{z\in U\colon\pmetric(z,\bd U)\geq \pmetric(V_m,\bd U)\}\subset U_{j_{m+1}}$ and set $V_{m+1}:=U_{j_{m+1}}$,
which yields the desired properties \eqref{bsineq} and $\max_{z\in\bd V_{m+1}}\pmetric(z,\bd U)<\pmetric(V_m,\bd U)$. Thus we have constructed the appropriate sets $\{V_k\}_{1\leq k\leq m+1}$.

Repeating this process we obtain the subsequence $\{V_k\}_{k\geq 1}$ satisfying \eqref{bsineq} for all $k\geq 2$. For $k\geq 1$ let $L_k:=\clos{V_k}$ and for 
$k\geq 2$ let $K_k:=(\hat L_{k-1})_{\hol(V_{k+1})}$, so $K_k$ is a compact subset of $V_{k+1}$ by holomorphic convexity of $V_{k+1}$. Let $k\geq 2$ and suppose
$z\in\bd V_k$. Then by \eqref{bsineq} $\pmetric(z,\bd V_{k+1})<\pmetric(V_{k-1},\bd V_{k+1})=\pmetric(L_{k-1},\bd V_{k+1})=\pmetric(K_k,\bd V_{k+1})$ (for the last equality we have
used Lemma~\ref{hulldist}), so $z\not\in K_k$. That is, $K_k\cap \bd V_k=\emptyset$. Since $L_{k-1}\subset V_k$, Corollary~\ref{okaweilcor} implies that $K_k\subset V_k$.

Thus $K_k\subset K_{k+1}$ and $K_k\subset V_k\subset V_{k+1}\subset U$ for all $k\geq 2$, and clearly $U=\bigcup_{k\geq 2} \inter K_k$, so by Lemma~\ref{bslemma} to show $U$ is a domain of holomorphy
it is enough to show $K_k$ is Runge in $K_{k+1}$ for all $k\geq 2$. But $K_k$ coincides with its $\hol(V_{k+1})$-convex hull, so by the Oka-Weil approximation theorem
$K_k$ is Runge in $V_{k+1}$, and because $K_{k+1}\subset V_{k+1}$ it follows that $K_k$ is Runge in $K_{k+1}$. Therefore $U$ is a domain of holomorphy.

Now suppose $U$ is unbounded, and let $z_0\in U_1$ and $r>0$. For each $j\geq 1$ let $W_j$ be 
the component of $U_j\cap \pdisc(z_0,r)$ containing $z_0$, so $W_j\subset W_{j+1}$ for all $j\geq 1$, each $W_j$ is a domain of holomorphy by Proposition~\ref{intscthol},
and $\bigcup_{j\geq 1} W_j$ is the component of $U\cap \pdisc(z_0,r)$ containing $z_0$. The above argument shows this (bounded) component is a domain of holomorphy, and this is true for each $r>0$,
so by Lemma~\ref{berslemma} we conclude that $U$ is a domain of holomorphy.
\end{proof}
Observe that the proof of this relatively simple fact about domains of holomorphy is long and complicated, while the corresponding result for pseudoconvex domains
is almost trivial (see Lemma~\ref{incps2}). This indicates the power of a more easily verified characterisation of domains of holomorphy -- if we had already solved the Levi
problem, the Behnke-Stein theorem would follow immediately from Lemma~\ref{incps2}. 
\subsection{Construction of unbounded functions}
Finally we are in a position to complete the solution to the Levi problem. First we show that
smoothly bounded strictly pseudoconvex domains are domains of holomorphy, and for this it is enough (by Theorem~\ref{usefuldh}) to show that at every boundary point of such a domain
there is a function holomorphic on the domain and tending to $\infty$ at that point. The following result ensures that such functions always exist:
\begin{proposition}
\label{globalbarrier}
Let $U\subset\C^n$ be a smoothly bounded strictly pseudoconvex domain  and let $a\in\bd U$. Then there exists $f\in\hol(U)$ tending to $\infty$ at $a$.
\end{proposition}
To prove this we require an intermediate result, which in turn requires Taylor's theorem expressed in a particular form.
Let $V\subset\C^n$ be open, $a\in V$ and $f\colon V\to \R$ a $\cts^2$ function. Regarding $f$ as a function of real variables, applying Taylor's theorem, and then writing 
the derivatives with respect to real variables in terms of those with respect to complex variables, we find that
 \begin{equation}
  f(z)=f(a)+2\Re\sprod{z-a}{\con{\cgrad{f}(a)}}+\Re \left(\Lambda_{z-a} f(a)\right)+\Delta_{z-a} f(a)+r(z), \label{dfexp} \end{equation}
 where $\Lambda_{z-a} f(a):=\sum_{j,k=1}^n \frac{\partial^2f}{\partial z_j\partial z_k}\big|_a(z_j-a_j)(z_k-a_k)$ and $r(z)/|z-a|^2\to 0$ as $z\to a$.
\begin{lemma}
\label{spsunbdedloc}
Let $U\subset\C^n$ be a smoothly bounded strictly pseudoconvex domain and let $a\in\bd U$. Then there is a function $g\in\hol(\C^n)$ and a neighbourhood $V$ of $a$ such that
$g(a)=0$ and $g(z)\neq 0$ for $z\in\clos{U}\cap V\setminus\{a\}$.
\end{lemma}
\begin{proof}
Let $f\colon W\to\R$ be a $\cts^\infty$ strictly plurisubharmonic global defining function for $U$, and for $z\in\C^n$ set $g(z):=2\sprod{z-a}{\con{\cgrad{f}(a)}}+\Lambda_{z-a} f(a)$,
so clearly $g\in\hol(\C^n)$ and $g(a)=0$. By Taylor's theorem we have $f(z) = \Re(g(z))+\Delta_{z-a} f(a) + r(z)$, where $r(z)/|z-a|^2\to 0$ as $z\to a$. Since $f$ is strictly plurisubharmonic we know
that $\Delta_\delta f(a)>0$ for all $\delta\in\C^n\setminus\{0\}$, and since $\delta\mapsto \Delta_\delta f(a)$ is continuous and $\{\delta\in\C^n\colon|\delta|=1\}$ is compact there is some $M>0$ such that
$\Delta_\delta f(a)\geq M$ for all $\delta\in\C^n$ with $|\delta|=1$. Since $r(z)/|z-a|^2\to 0$ as $z\to 0$ there is a neighbourhood $V$ of $a$ such that $|r(z)|/|z-a|^2<M/2$ for $z\in V$.

Suppose $z\in \clos{U}\cap V\setminus\{a\}$. Then
\begin{align*}
 0\geq f(z) = \Re(g(z)) + \Delta_{z-a} f(a) + r(z) &\geq \Re(g(z)) + |z-a|^2 M - |z-a|^2M/2\\
	&= \Re(g(z))+|z-a|^2 M/2,
\end{align*}
and since $|z-a|^2M/2>0$ it follows that $\Re(g(z))<0$ and thus $g(z)\neq 0$, as required.
\end{proof}
Now we can prove Proposition~\ref{globalbarrier}.
\begin{proof}[Proof of Proposition~\ref{globalbarrier}]
By Lemma~\ref{spsunbdedloc} there exists a function $g\in\hol(\C^n)$ and a neighbourhood $V$ of $a$ such that $g(a)=0$ and $g(z)\neq 0$ for $z\in \clos{U}\cap V\setminus\{a\}$.
Passing to a subset of $V$ we may assume that $g(z)\neq 0$ for $z\in\clos{U}\cap \clos{V}\setminus\{a\}$, so in particular there is a neighbourhood $Y$
of $\clos{U}\setminus\{a\}$ such that $g$ is non-vanishing on $V\cap Y$.
Let $W$ be a bounded neighbourhood of $a$ such that $\clos{W}\subset V$, and replace $Y$ with $Y\cup W$, so now $Y$ is a neighbourhood of $\clos{U}$ and $g$
is non-vanishing on $V\cap (Y\setminus W)$.
 
Let $\chi\colon\C^n\to\R$ be a $\cts^\infty$ function compactly supported in $V$ and satisfying $\chi(z)=1$ for $z\in W$. 
Let $\omega:=\frac{\con\partial\chi}{g}$, so because $\con\partial\chi$ is zero on a neighbourhood of $W\cup V^c$ and $g$ is non-vanishing 
on $V\cap (Y\setminus W)$ we see that $\omega$ is a smooth $(0,1)$-form on $Y$. Furthermore, since $g$ is holomorphic we have $\con\partial\omega=0$.

Now
we will replace $Y$ with a smoothly bounded strictly pseudoconvex domain. Let $g\colon Z\to\R$ (where $Z$ is a neighbourhood of $\clos{U}$) be a $\cts^\infty$ strictly plurisubharmonic global defining
function for $U$, so $g^{-1}((-\infty,\epsilon))\subset Y$ for sufficiently small $\epsilon>0$. 
In view of Lemmas~\ref{morselemma} and \ref{globdefworks} by decreasing $\epsilon$ if necessary we may assume
$z\mapsto g(z)-\epsilon$ is a global defining function for the component of $g^{-1}((-\infty,\epsilon))$ containing $U$.
Replace $Y$ with this component,
so $Y$ is now a smoothly bounded strictly pseudoconvex domain such that $\omega$ is smooth on $Y$ and
$\clos{U}\subset Y$.

By H\"ormander's theorem there exists $\phi\in\ccts^\infty(Y)$ with $\con\partial\phi = \frac{\con\partial\chi}{g}$, so the function $f\colon U\to\C$ given by
$f(z):=\frac{\chi(z)}{g(z)}-\phi(z)$ is $\ccts^\infty$ (as $g$ is non-vanishing on $U$), and since $\con\partial f = \frac{\con\partial\chi}{g}-\con\partial\phi=0$ it follows that $f\in\hol(U)$.
As $z\to a$ (with $z\in U$) we have $g(z)\to 0$, $\chi(z)\to 1$ and $\phi(z)\to\phi(a)\in\C$, so $f(z)\to\infty$. Thus $f$ tends to $\infty$ at $a$, as required.
\end{proof}
In conjunction with Theorem~\ref{usefuldh}, this immediately implies:
\begin{corollary}
\label{spsdh}
If $U\subset\C^n$ is a smoothly bounded strictly pseudoconvex domain then $U$ is a domain of holomorphy.
\end{corollary}
The required result now follows easily:
\begin{theorem}
If $U\subset\C^n$ is a pseudoconvex domain then $U$ is a domain of holomorphy.
\end{theorem}
\begin{proof}
By Proposition~\ref{slpapprox} there exists a sequence $\{U_j\}_{j\geq 1}$ of smoothly bounded strictly pseudoconvex domains 
such that $U_j\subset U_{j+1}$ for all $j\geq 1$ and $U=\bigcup_{j\geq 1}U_j$, and by Corollary~\ref{spsdh} each $U_j$ is a domain of holomorphy, so by the Behnke-Stein theorem $U$ is a domain
of holomorphy.
\end{proof}
With the results of previous sections in mind we see that the domains of holomorphy, holomorphically convex domains, pseudoconvex domains,
locally pseudoconvex domains and domains satisfying the continuity principle are
identical. In the case of domains with twice-differentiable boundaries, these domains are precisely the Levi pseudoconvex domains. This completes the solution to the Levi problem.
\index{Levi problem}
\subsection{Generalisations}
\label{sec:concl}
Here we have considered the Levi problem and its related concepts 
in their simplest setting, namely for domains in $\C^n$. Such domains are the most basic type of \emph{complex manifold} -- these are
topological manifolds which admit atlases with biholomorphic transition maps, and on these spaces one may naturally define the notion of a holomorphic function
(see, for example, \cite[subsection 12]{itca}). The natural generalisation of a domain
of holomorphy to general complex manifolds is a \emph{Stein manifold}, which is a complex manifold satisfying holomorphic convexity, defined as in Section~\ref{chap:hol}, and
holomorphic separability, which means that for any two points in the manifold there is a holomorphic function attaining different values at the points (see \cite[page 223]{itca}
for an introduction).
The Levi problem
in this setting asks for necessary and sufficient conditions for a complex manifold to be Stein. In 1953 Oka solved the Levi problem for Riemann domains spread over $\C^n$ (which are special
types of complex manifold) \cite{oka}, and in 1958 Grauert generalised this result to general complex manifolds \cite{grauert58}.  
This was further generalised by Narasimhan in 1961 who proved that a complex space is a Stein space if and only if it admits a strictly plurisubharmonic exhaustion function \cite{narasimhan}. 
A \emph{complex space} is a generalisation of a complex manifold which allows the presence of singularities, and a \emph{Stein space} is the corresponding generalisation of a Stein manifold (see
chapters 5 and 7 of
the textbook \emph{Analytic Functions of Several Complex Variables} by Gunning and Rossi \cite{gunning} for an introduction). In the intervening years solutions to the Levi problem
for even more general spaces have been presented -- see, for example, the article \cite{siu} for a more detailed survey of these results and their implications.

\section*{Acknowledgements}
The author would like to thank Alexander Isaev for his extensive guidance and feedback.

\nocite{henlei,henkinbook}
\begin{flushleft}
\bibliographystyle{abbrv}
\bibliography{bib_extra}
\end{flushleft}
{\obeylines
\noindent Department of Quantum Science, The Australian National University, Canberra, 

\noindent ACT 0200, Australia, e-mail: {\tt harry.slatyer@anu.edu.au}}
\printindex
\end{document}